\documentclass[12pt,english]{article}
\usepackage[T1]{fontenc}
\usepackage{geometry}
\usepackage{amsmath}
\usepackage{amssymb}
\usepackage{amsthm}
\usepackage{color}
\usepackage{graphicx}
\usepackage{epstopdf}
\usepackage{subfigure}
\usepackage{float}
\usepackage[noend]{algpseudocode}
\usepackage{algorithmicx,algorithm}
\usepackage{authblk}
\makeatletter
\@ifundefined{date}{}{\date{}}

\newcommand{\rom}[1]{\uppercase\expandafter{\romannumeral#1}}

\newcommand{\beq}{\begin{equation}}
\newcommand{\eeq}{\end{equation}}
\newcommand{\bal}{\begin{align}}
\newcommand{\eal}{\end{align}}
\newcommand{\baln}{\begin{align*}}
\newcommand{\ealn}{\end{align*}}

\theoremstyle{plain}\newtheorem{theorem}{Theorem}[section]
\theoremstyle{definition}\theoremstyle{plain}\newtheorem{corollary}{Corollary}
\newtheorem{proposition}{Proposition}[section]
\newtheorem{example}{Example}[section]

\newcommand{\bm}{\boldsymbol}

\graphicspath{{figure/}}

\makeatother

\usepackage{babel}
\begin{document}

\title{Asymptotically stable Particle-in-Cell methods for the magnetized Vlasov--Poisson equations in orthogonal curvilinear coordinates}

\author[1]{Anjiao Gu} 
\author[2,3]{Yajuan Sun \thanks{Corresponding author: sunyj@lsec.cc.ac.cn.}} 

\affil[1]{Department of mathematics, Institute of Natural Sciences and MOE-LSC. Shanghai Jiao Tong University, Shanghai, 200240, China}

\affil[2]{State Key Laboratory of Mathematical Sciences, Academy of Mathematics and Systems Science, Chinese Academy of Sciences, Beijing 100190, China}

\affil[3]{School of Mathematical Sciences, University of Chinese Academy of Sciences, Beijing, 100190, China}

\renewcommand*{\Affilfont}{\small\it}
\renewcommand\Authands{ and }

\maketitle

\abstract{
In high-temperature plasma physics, strong magnetic fields are essential for confining charged particles.
Consequently, classical mathematical models of such systems must take into account the external magnetic field effects.
A key governing equation is the magnetized Vlasov--Poisson system, which exhibits multiscale dynamics and rich physical properties.
Thus, developing the structure-preserving numerical methods and studying its capability of maintaining these intrinsic properties over long-time simulations is therefore critically important.
This paper presents a general framework for constructing and analyzing structure-preserving methods in orthogonal curvilinear coordinates.
We prove that in these coordinates the Poisson-bracket structure is retained under appropriate finite element discretizations.
However, the resulting Hamiltonian systems in transformed coordinates typically can't be decomposed as several subsystems which can be solved exactly.
To address this, we propose a semi-implicit numerical scheme that can still maintain the  favorable stability properties inherited by the system.
The effectiveness of the new derived numerical methods is demonstrated in application to strongly magnetized systems, and the rigorous asymptotic stability analysis are provided.


\section{Introduction}

Plasma physics fundamentally involves the dynamical study of charged  multi-particle  interacting with electromagnetic fields.
Typically, these systems are treated classically, neglecting relativistic effects.
At the microscopic scale, the kinetic behavior of individual charged particles in electromagnetic fields is governed by Newtonian mechanics.
However, for macroscopic plasma dynamics where particle densities become  high and the collisions become significant,  a conducting fluid description proves more appropriate.
This constitutes the framework of magnetohydrodynamics.
When particle velocity distributions require explicit resolution, kinetic theory becomes essential.
The Vlasov equation emerges as a key model in this regime, providing a statistical description of collisionless plasma dynamics.

Plasma systems exhibit far more complex physical properties than ordinary fluids, with broad applications across disciplines.
This complexity makes the investigation of their kinetic behavior both challenging and critically important.
While numerical algorithms for the Vlasov equation have been studied extensively since the 1960s, structure-preserving methods leveraging its geometric properties have emerged as a significant research focus only in recent decades~\cite{Filbet2003N,Morrison2016S,XJY2018S}. These geometric algorithms are particularly valuable for long-term particle simulations, as they excel at preserving the key conservative quantities of the original systems~\cite{FK1985,FK2010}.
The rich geometric structure of Vlasov systems naturally motivates the development of structure-preserving numerical methods.
For instance, the Vlasov--Maxwell system possesses both variational~\cite{Squire2012G,Evstatiev2013V,Martin2021Variational} and Hamiltonian structures ~\cite{Morrison1980T,Marsden1982T,Morrison2013A}. The two formulations  can  associate directly to a  set of conserved quantities.

While the majority of existing work on geometric algorithms for the Vlasov equation has focused on Cartesian coordinate systems~\cite{Nicolas2015H,XJY2015E,Qin2016C,HY2016,Kraus2017,Casas2017,LYZ2019}, many physical problems are more naturally characterized by a configuration-specific coordinates.
In magnetic confinement fusion, for instance, the electromagnetic fields are typically described using toroidal coordinates.
Such coordinate systems often reveal important symmetries in physical quantities that are not immediately apparent in Cartesian representations.
This observation motivates our effort to extend both geometric formulations and their corresponding numerical frameworks from Cartesian to general curvilinear coordinates.
However, the construction of structure-preserving algorithms in curvilinear coordinates presents non-trivial challenges that cannot be resolved through simple coordinate transformations of existing Cartesian-based methods.
The development of such algorithms requires fundamentally new approaches to preserve geometric structures under coordinate transformations.
Within curvilinear coordinate frameworks, particle methods have been well-established in the literature~\cite{Hockney1988, Birdsall1991}.
More recently, these approaches have been successfully extended to the Vlasov--Maxwell system~\cite{Eric2021Geometric,Eric2021Perfect}, demonstrating the feasibility of structure-preserving methods in non-Cartesian settings.
On the other hand, the simulation of kinetic-scale plasma physics poses a fundamental challenge due to the system's intrinsic multi-scale characteristics.
Key physical scales - including the Debye length, particle gyroradius, and collision mean free paths - typically span several orders of magnitude relative to macroscopic device dimensions. 
Conventional explicit numerical methods become computationally intractable as they mandate resolution of all microscopic scales to ensure stability, rendering large-scale or extended-duration simulations prohibitively expensive.
Hence, multiscale numerical methods for particle dynamics simulation have become a focus of extensive research~\cite{BWZ2014U,JS2017T,Chartier2020,Hairer2020A,Katharina2020,Bacchini2020,Hairer2021L}.
A widely recognized solution framework is provided by asymptotic-preserving (AP) schemes~\cite{larsen1987asymptotic,golse1999convergence,jin1993fully,Filbet2016,zhu2017vlasov,crouseilles2013asymptotic,degond2010asymptotic,blaustein2024structure}, , which establish a unified computational methodology.
These methods are specifically designed to maintain the asymptotic transition properties between microscopic and macroscopic models at the discrete level.
Through their inherent adaptive capability, AP schemes automatically adjust their resolution between fine-scale kinetic descriptions and coarse-scale continuum representations, offering an elegant solution  in multiscale computational challenges.
For detailed discussions on the theoretical foundations and practical implementations of AP methods, we refer readers to the comprehensive review by Jin~\cite{Jin2022Asymptotic}.

In this paper, we discuss the structure-preserving algorithms for magnetized Vlasov--Poisson equations described in orthogonal curvilinear coordinates.
We have proven that the solution of the variational problem in orthogonal curvilinear coordinates  exists and unique, thus the finite element method is still suitable to be used.
In the new coordinate, though the system still has the structural properties, the Hamiltonian splitting technique is no longer applicable in order to derive the structure-preserving methods.
For the system in new formulation we apply the semi-implicit method which studied in~\cite{Filbet2016,Filbet2023Asymptotically}, and have proven that the fully discretization is asymptotically stable.
The main results of this paper can be summarized as follows.
\begin{itemize}
\item Generalization of Hamiltonian particle methods for magnetized Vlasov--Poisson systems in general curvilinear coordinate systems
\item A novel asymptotic-preserving temporal discretization algorithm is proposed and analyzed for two-dimensional problems
\end{itemize}

The outline of the paper is as follows.
In section 2, we introduce the magnetized Vlasov--Poisson equation.
Then in section 3 we present the equation in orthogonal curvilinear coordinates and construct the spatial discretization with which the semi-discrete system maintains the Hamiltonian structure inherited by the original system.
Furthermore, we discuss the temporal discretizations, and prove the asymptotically stable property of the resulting numerical methods in Section 4.
Section 5 demonstrates the numerical results.
Finally, we conclude this paper.

\section{The magnetized Vlasov--Poisson equation}

For the research of magnetic confinement fusion, it is natural and necessary to introduce the effect of an external magnetic field.
We consider a plasma consisting of a large number of charged particles while its distribution is described by the Vlasov equation.
The self-consistent electric field expressed by the potential satisfies the Poisson equation.
The magnetized Vlasov--Poisson equation can be expressed as
\begin{equation}\label{VP}
\left\{
\begin{aligned}
&\frac{\partial f}{\partial t}+v\cdot\frac{\partial f}{\partial x}+(\bm{E}+v\times \bm{B})\cdot\frac{\partial f}{\partial v}=0,\\
&\bm{E}=-\nabla_x\phi,\quad\nabla_x\cdot \bm{B}=0,\\
&-\Delta{\phi}=\rho(x,t)-\rho_0,\quad\rho=\int_{\Omega_{v}} fdv,
\end{aligned}
\right.
\end{equation}
where $f(t,x,v)$ is the distribution function of the particle, $x\in\Omega_{x}\subset\mathbb{R}^{3}$, $v\in\Omega_{v}\subset\mathbb{R}^{3}$ and $t\in\mathbb{R}_{+}$ denotes position, velocity and time in turn.
$\bm{E},\bm{B}\in\mathbb{R}^3$ are the self-consistent electric field and external magnetic field respectively.
The initial distribution satisfies $f(x,v,0)=f_0(x,v)$.
Furthermore, function $\rho$ denotes the density of the charge, and the ions are assumed to be homogeneous and their density is $\rho_0$.

For later use, we denote the $0$-th, $1$-st and $2$-nd moments of distribution function by
$$\rho(t,x)=\int_{\Omega_{v}} f(t,x,v)dv,\ J(t,x)=\int_{\Omega_{v}}v f(t,x,v)dv$$
and $$S(t,x)=\int_{\Omega_{v}}v\otimes v f(t,x,v)dv=\int_{\Omega_{v}}vv^\top f(t,x,v)dv.$$
In practical computation, the boundary conditions can be taken as zero, i.e., $f(t,x,v)=0$ and $\phi(x)=0$ on $\partial\Omega_x$ while the periodic boundary condition is also used frequently.
As an important model for studying the behavior of magnetically-confined plasma, the magnetized Vlasov--Poisson system  has many conservation properties.
By integrating the Vlasov equation in $v$ yield, one can get the following propositions.
\begin{proposition}
MVP equation \eqref{VP} has the following conservation laws:

Continuity equation:
\begin{equation}\label{ContinuityEq}
\partial_t \rho+\nabla_x\cdot J=0.
\end{equation}
	
Moment equation:
\begin{equation}\label{MomentEq}
\partial_t J+\nabla_x\cdot S-\rho\bm{E}-J\times\bm{B}=0.
\end{equation}
\end{proposition}

The conservative properties of the MVP equation \eqref{VP} can be described as the following proposition.
\begin{proposition}
If the potential function $\phi$ has a zero boundary or a periodic boundary, $f$ is periodic in $x$ and is compactly supported in $v$, the following quantities are all conserved.
	
(a) {\bf{Charge}}: $\mathcal{Q}=\int_{\Omega} fdxdv=\int_{\Omega_x} \rho dx$.
	
(b) {\bf{Energy}}: $\mathcal{H}=\frac{1}{2}\int_{\Omega} \vert v\vert^2 fdxdv+\frac{1}{2}\int_{\Omega_x} \vert \bm{E}\vert^2 dx$.
	
(c) {\bf{Entropy}}: $\mathcal{S}=\int_{\Omega} f\ln{f} dxdv$.
\end{proposition}

In our previous work~\cite{GAJ2022Hamiltonian}, we have presented \eqref{VP} possesses the following Poisson bracket
\begin{equation}
\begin{aligned}
\{\{\mathcal{F},\mathcal{G}\}\} (f)
&=\int_{\Omega} f\left\{ \frac{\delta\mathcal{F}}{\delta f},\frac{\delta\mathcal{G}}{\delta f}\right\}_{xv}dxdv\\
&+\int_{\Omega} f\bm{B}\cdot\left(\frac{\partial}{\partial v}\frac{\delta\mathcal{F}}{\delta f}\times\frac{\partial}{\partial v}\frac{\delta\mathcal{G}}{\delta f}\right)dxdv,
\end{aligned}
\label{MMWB}
\end{equation}
where $\mathcal{F}$ and $\mathcal{G}$ are two functionals of $f$, $\frac{\delta \mathcal{F}}{\delta f}$ is the variational derivative, $\Omega:=\Omega_{x}\times\Omega_{v}$. In \eqref{MMWB}, the operator $\left\{ \cdot,\cdot\right\} _{xv}$ is the canonical Poisson bracket which for two given functions $m(x,v)$ and  $n(x,v)$, that is
\[
\{m,n\}_{xv}=\frac{\partial m}{\partial x}\cdot\frac{\partial n}{\partial v}-\frac{\partial m}{\partial v}\cdot\frac{\partial n}{\partial x}.
\]
With the bracket \eqref{MMWB}, we consider the following Poisson system
\begin{equation}
\frac{d\mathcal{F}}{d t}=\{\{\mathcal{F},\mathcal{H}\}\},\label{PoissonVM}
\end{equation}
where $\mathcal{F}$ is any functional of $f$.
In fact, by setting
$$\mathcal{F}[f]=\int_{\Omega}f(\tilde{x},\tilde{v},t)\delta(x-\tilde{x})\delta(v-\tilde{v})d\tilde{x}d\tilde{v},$$
and defining the local energy
$$h(x,v)=\frac{\delta\mathcal{H}}{\delta f}(x,v)={\vert v\vert^{2}}/2+\phi({x}),$$
it follows from \eqref{PoissonVM} that
\begin{align*}
\frac{\partial f}{\partial t}&=-\{f,h\}_{xv}-\bm{B}\cdot\left(\frac{\partial f}{\partial v}\times \frac{\partial h}{\partial v}\right)
\end{align*}
which recovers the MVP equation \eqref{VP}.
This implies that the MVP equation can be written in a Poisson system.
Through Poisson bracket \eqref{MMWB} and expression \eqref{PoissonVM}, the above proposition can be easily verified~\cite{GAJ2022Hamiltonian}.

\section{Curvilinear coordinates}

Equations can be expressed in different forms when using corresponding coordinate systems.
However, the formulations of them in Cartesian coordinate are usually not the simplest.
For example, the spherical equation $x_1^2+x_2^2+x_3^2=R^2$ can be written as $r=R$ in spherical coordinate $(r,\varphi,\theta)$.
Furthermore, the results in~\cite{Filbet2023Asymptotically,Shi2024Drift} extend time-discretization methods to fully three-dimensional magnetic fields exhibiting toroidal symmetry. 
This geometric framework models realistic configurations employed in tokamak devices.
Therefore, studying equations in curvilinear coordinates has broad significance.
Vector fields in orthogonal curvilinear coordinates we used are given in the appendix~\ref{app:0}.

\subsection{Transformation}
We consider the bijective coordinate transformation from the space $\tilde{\Omega}_y$ to $\Omega_x$. The transformation is denoted by
$$F:\tilde{\Omega}_y\to\Omega_x\subset\mathbb{R}^3,\quad F(y)=x.$$
The Jacobi matrix is denoted by $D_F(y)$, and its elements are
$$\left(D_F(y)\right)_{ij}=\frac{\partial x_i}{\partial y_j},$$
the corresponding Jacobian is $\mathbb{J}(y)=\det(D_F(y))$.
Moreover, we assume that the new coordinates are orthogonal and $\mathbb{J}(y)>0$.
If $\mathbb{J}(y)<0$, we only need to apply the permutation of $y_i$ and $y_j$ once so that the new Jacobian $\mathbb{J}(y')>0$ about the transformation $y'$.
Also, $\mathbb{J}(y)\neq 0$ due to the orthogonality.
Then we can denote $D_F(y)^{-T}$ as $N(y)$. The curvilinear coordinates to the differential forms have been introduced as the following proposition.

\begin{proposition}\label{transform}
For a scalar differential 0-form $g\in H^{1}(\Omega_x)$, we can define $\tilde{g}\in H^{1}(\tilde{\Omega}_y)$ as
$$\tilde{g}(y)=g(x).$$

For a vector differential 1-form $\bm{E}\in H(curl,\Omega_x)$, it has $\tilde{\bm{E}}\in H(curl,\tilde{\Omega}_y)$ connected by
$$\bm{E}(x)=N(y)\tilde{\bm{E}}(y).$$

For a vector differential 2-form $\bm{B}\in H(div,\Omega_x)$, we have $\tilde{\bm{B}}\in H(div,\tilde{\Omega}_y)$ through
$$\bm{B}(x)=\frac{D_F(y)}{\mathbb{J}(y)}\tilde{\bm{B}}(y).$$

For a scalar differential 3-form $h\in L^{2}(\Omega_x)$, we can obtain $\tilde{h}\in L^{2}(\tilde{\Omega}_y)$ via
$$h(x)=\frac{1}{\mathbb{J}(y)}\tilde{h}(y).$$
\end{proposition}

The proof could be found in~\cite{Eric2021Geometric} which also can be deduced from Appendix~\ref{app:0}.
Then we can have $\nabla_{x}\cdot \bm{B}=\frac{1}{\mathbb{J}(y)}\nabla_{y}\cdot \tilde{\bm{B}}$ and $\nabla_x=N(y)\nabla_y$.
According to Proposition~\ref{transform}, we can transform the magnetized Vlasov--Poisson system to
\begin{equation}\label{CVP}
\left\{
\begin{aligned}
&\frac{\partial \tilde{f}}{\partial t}+N(y)^{T}v\cdot\frac{\partial \tilde{f}}{\partial y}+N(y)(\tilde{\bm{E}}+(N(y)^{T}v)\times \tilde{\bm{B}})\cdot\frac{\partial \tilde{f}}{\partial v}=0,\\
&\nabla_{y}\cdot \tilde{\bm{B}}=0,\quad -N(y)\nabla_y\cdot N(y)\nabla_y \tilde{\phi}=\int_{\mathbb{R}^3} \tilde{f}dv-\rho_0.
\end{aligned}
\right.
\end{equation}
Specifically, in this model we consider long time behavior of the plasma and a strong magnetic field.
Both of them are scaled by parameter $\varepsilon$.

\noindent{\bf Remark.} It is worth mentioning that $\tilde{f}$ is no longer conserved, but $\mathbb{J}\tilde{f}$ is.
In other words,
\begin{equation*}
\frac{\partial \mathbb{J}\tilde{f}}{\partial t}+\nabla_{y}\cdot(N^\top v\mathbb{J}\tilde{f})+\nabla_{v}\cdot(N(\tilde{\bm{E}}+(N^\top v)\times \tilde{\bm{B}})\mathbb{J}\tilde{f})=0.
\end{equation*}
Some corresponding conservative methods can been found in~\cite{Chacon2016A,Hamiaz2016The}.

We use Particle-in-cell method to obtain a discrete distribution function which reads
\begin{equation*}
f_h(x,v,t)=\sum\limits_{s=1}^{N_p}\alpha_s\delta(x-X_s(t))\delta(v-V_s(t)).
\end{equation*}
Here, $N_p$ is the number of particles while $\alpha_s$ is the weight of particle.

After the transformation, the distribution function becomes to
\begin{equation*}
\tilde{f}_h(y,v,t)=\sum\limits_{s=1}^{N_p}\alpha_s\frac{\delta(y-Y_s(t))}{\mathbb{J}(y)}\delta(v-V_s(t))
\end{equation*}
due to the Dirac function $\delta$.
To generate particles more conveniently, the following equivalent form can be used
\begin{equation*}
\mathbb{J}(y)\tilde{f}_h(y,v,t)=\sum\limits_{s=1}^{N_p}\alpha_s \delta(y-Y_s(t))\delta(v-V_s(t)).
\end{equation*}

Thus Vlasov equation in (\ref{CVP}) transforms into the following particle equation
\begin{equation}\label{CCP1}
\left\{
\begin{aligned}
&\dot{Y_s}=N(Y_s)^{T}V_s,\\
&\dot{V_s}=N(Y_s)\tilde{\bm{E}}(Y_s)+N(Y_s)\hat{\tilde{\bm{B}}}(Y_s)N(Y_s)^{T}V_{s}, \quad
 s=1,2,\cdots ,N_p.
\end{aligned}
\right.
\end{equation}

\subsection{Spatial discretization and Hamiltonian structure}

In~\cite{GAJ2022Hamiltonian}, we have discretized the space with finite element method to ensure that the discrete potential $\phi_h\in H_0^{1}(\Omega_x)$.
Thus the semi-discrete system can preserves the Poisson structure.
In the orthogonal curvilinear coordinate, the discrete potential is denoted by $\tilde{\phi}_h\in H_0^{1}(\tilde{\Omega}_y)$.

In order to employ finite element discretization in space, we need to determine the well-defined variational problem.
In the following theorem, we have formed the variational formulation in orthogonal curvilinear coordinate and have proven that the corresponding bilinear form  obey the conditions of Lax-Milgram theorem.
That is, it is  bounded and coercive.

\begin{theorem}
The bilinear form $a'(u,v)=-\int_{\tilde{\Omega}} (N(y)\nabla_y\cdot N(y)\nabla_y u)v \mathbb{J}(y)dy$ is bounded and coercive.
\end{theorem}

\begin{proof}
By noticing that
$$N(y)\nabla_y\cdot N(y)\nabla_y u=\frac{1}{H_1 H_2 H_3}\sum_{ijk}\frac{\partial}{\partial y_i}(\frac{H_j H_k}{H_i} \frac{\partial u}{\partial y_i})$$
where $H_1,H_2,H_3$ are Lame coefficients, it's obvious since $H_i>0$ and $H_1 H_2 H_3=\mathbb{J}(y)$.
\end{proof}

We assume a finite dimensional space $\tilde{V}_{h}$ is the subspace of $H_0^{1}(\tilde{\Omega}_y)$ and $\{W_j(y)\}_{j=1}^{N}$ are piecewise polynomial basis functions of it.
The discrete potential function $\tilde{\phi}_h\in \tilde{V}_{h}$ can be expressed as
\begin{equation}\label{CCphi}
\tilde{\phi}_h(y,t)=\sum_{j=1}^{N}\tilde{\phi}_{j}(t){W}_{j}(y).
\end{equation}

Substituting (\ref{CCphi}) into the particle equation (\ref{CCP1}), then we have
\begin{equation}\label{CCP2}
\left\{
\begin{aligned}
&\dot{Y_s}=N(Y_s)^{T}V_s,\\
&\dot{V_s}=-N(Y_s)\sum_{j=1}^{N}\tilde{\phi}_{j}\nabla{W}_{j}(Y_s)+N(Y_s)\hat{\tilde{\bm{B}}}(Y_s)N(Y_s)^{T}V_{s}.
\end{aligned}
\right.
\end{equation}

We can represent the equation with the following discrete Poisson bracket (\ref{CCP2}).
For any functions $F$ and $G$ on $(Y,V)$, define the bracket
\begin{equation}\label{CCbra}
\begin{aligned}
\left\{ F,G\right\} (Y,V)
&=\sum_{s=1}^{N_p}\frac{1}{\alpha_{s}}\left(N(Y_s)\frac{\partial F}{\partial Y_{s}}\cdot\frac{\partial G}{\partial V_{s}}-N(Y_s)\frac{\partial G}{\partial Y_{s}}\cdot\frac{\partial F}{\partial V_{s}}\right)\\
&+\sum_{s=1}^{N_p} \frac{1}{\alpha_{s}}\frac{D_F(Y_s)}{\mathbb{J}(Y_s)}\tilde{\bm{B}}(Y_s)\cdot\left(\frac{\partial F}{\partial V_s} \times\frac{\partial G}{\partial V_s}\right).
\end{aligned}
\end{equation}
and discrete Hamiltonian function
\begin{equation}\label{CCHam}
H(Y,V)=\frac{1}{2} \sum_{s=1}^{N_p} \alpha_s \vert V_s\vert^{2}+\frac{1}{2}\sum_{j=1}^N\sum_{k=1}^N \phi_\mathbb{J}(y)\phi_k(Y) a'(W_j,W_k).
\end{equation}
By means of discrete bracket (\ref{CCbra}) and discrete Hamiltonian  (\ref{CCHam}), the semi-discrete system (\ref{CCP2}) can be written as
\begin{equation*}
\dot{Y}_s=\{{Y}_s,H\}, \quad \dot{V}_s=\{{V}_s,H\}.
\end{equation*}

In the following we express the semi-discrete system (\ref{CCP2}) in matrix form.
First we introduce the following notations:
\begin{equation*}
\begin{gathered}
\tilde{\Phi}(Y)=(\tilde{\phi}_1,\tilde{\phi}_2,\ldots,\tilde{\phi}_{N})^{T}(Y) \in \mathbb{R}^{N},\quad \mathbb{M}\in \mathbb{R}^{N\times N},\quad \mathbb{M}_{jk}=a'(W_j,W_k),\\
\mathbb{G}(Y)=\left(
          \begin{array}{cccc}
            \nabla W_1(Y_1) & \nabla W_2(Y_1) & \cdots & \nabla W_N(Y_1) \\
            \nabla W_1(Y_2) & \nabla W_2(Y_2) & \cdots & \nabla W_N(Y_2) \\
            \cdots & \cdots & \cdots &\cdots \\
            \nabla W_1(Y_{N_p}) & \nabla W_2(Y_{N_p}) & \cdots & \nabla W_N(Y_{N_p}) \\
          \end{array}
        \right)\in \mathbb{R}^{(3N_P)\times N},\\
\tilde{\mathbb{B}}(Y)={\rm{diag}}(\hat{\tilde{\bm{B}}}(Y_1),\hat{\tilde{\bm{B}}}(Y_2),\cdots,\hat{\tilde{\bm{B}}}(Y_{N_p}))\in \mathbb{R}^{(3N_p)\times(3N_p)},\\
\mathbb{N}={\rm{diag}}(N(Y_1),N(Y_2),\cdots,N(Y_{N_p}))\in \mathbb{R}^{(3N_p)\times(3N_p)}.
\end{gathered}
\end{equation*}

Define the diagonal weight matrix $\Omega=\mathrm{diag}(\alpha_1,\alpha_2,\ldots,\alpha_{N_p})\in \mathbb{R}^{N_p\times N_p}$ and $3$-dimensional identity matrix $I$.
Further we note that $$\mathbb{W}=\Omega\otimes I\in \mathbb{R}^{(3N_p)\times (3N_p)}.$$
Thus, the discrete Poisson bracket (\ref{CCbra}) can be rewritten in the following matrix form
\begin{equation}\label{dispb}
\left\{ F,G\right\}(Z)=\frac{\partial F}{\partial Z}^\top \mathbb{K}(Y) \frac{\partial G}{\partial Z},
\end{equation}
with $Z=(Y^\top,V^\top)^\top$, and
\begin{equation}\label{CCbramat}
\mathbb{K}(Y)=\left(
                \begin{array}{cc}
                  0 & \mathbb{W}^{-1}\mathbb{N}(Y)^\top \\
                  -\mathbb{W}^{-1}\mathbb{N}(Y) & \frac{1}{\varepsilon}\mathbb{W}^{-1}\mathbb{N}(Y)\tilde{\mathbb{B}}(Y)\mathbb{N}(Y)^\top \\
                \end{array}
              \right)
\end{equation}
is the Poisson matrix corresponding to the bracket (\ref{CCbra}).
After using the above notations, the discrete Hamiltonian (\ref{CCHam}) can be written as
\begin{equation}
\begin{aligned}
H(Y,V)=\frac{1}{2} V^\top\mathbb{W}V+\frac{1}{2}\tilde{\Phi}(Y)^{T} \mathbb{M} \tilde{\Phi}(Y).
\end{aligned}
\end{equation}
The corresponding system (\ref{CCP2}) can be rewritten in terms of the Poisson matrix $\mathbb{K}(X)$ as
\begin{equation}\label{disCVPODE}
\left\{
\begin{aligned}
\dot{Y}&=\mathbb{N}(Y)^\top{V},\\
\dot{V}&=\mathbb{N}(Y)\mathbb{G}(Y)\tilde{\Phi}(Y)+\mathbb{N}(Y)\tilde{\mathbb{B}}(Y)\mathbb{N}(Y)^\top V.
\end{aligned}
\right.
\end{equation}

\begin{theorem}
Discrete bracket~(\ref{dispb}) with respect to system (\ref{disCVPODE}) is a Poisson bracket.
\end{theorem}

\begin{proof}

According to proposition \ref{brapro}, we need to verify the matrix (\ref{CCbramat}) satisfies
\begin{equation*}
\sum_{l=1}^{6{N_p}}\left(\frac{\partial\mathbb{K}_{ij}}{\partial Z^l}\mathbb{K}_{lk}+\frac{\partial\mathbb{K}_{jk}}{\partial Z^l}\mathbb{K}_{li}+\frac{\partial\mathbb{K}_{ki}}{\partial Z^l}\mathbb{K}_{lj}\right)=0,
\end{equation*}
for all indexes $i,j,k\in\{1,\cdots,6N_p\}$ and $Z=(Y,V)$.

Since the matrix (\ref{CCbramat}) depends on $Y$, we only need to consider the case of $l\in\{1,\cdots,3N_p\}$.
That is,
\begin{equation}\label{braN1}
\sum_{l=1}^{3{N_p}}\left(\frac{\partial\mathbb{K}_{ij}}{\partial Y^l}\mathbb{K}_{lk}+\frac{\partial\mathbb{K}_{jk}}{\partial Y^l}\mathbb{K}_{li}+\frac{\partial\mathbb{K}_{ki}}{\partial Y^l}\mathbb{K}_{lj}\right).
\end{equation}
Without loss of generality, we suppose that $\frac{\partial\mathbb{K}_{ij}}{\partial Y^l}\mathbb{K}_{lk}\neq 0$.
When $k\in\{1,\cdots,3N_p\}$, it has $\mathbb{K}_{lk}=0$.
Thus, it only needs to focus on the case of $k\in\{3N_p+1,\cdots,6N_p\}$.

{\bf{Case 1.}} When $i,j\in\{1,\cdots,3N_p\}$, it has $\mathbb{K}_{ij}=0$.

{\bf{Case 2.}} When $i\in\{1,\cdots,3N_p\},j\in\{3N_p+1,\cdots,6N_p\}$, it has $\mathbb{K}_{li}=0$.
Then (\ref{braN1}) becomes to
\begin{equation}\label{braN2}
\sum_{l=1}^{3{N_p}}\left(\frac{\partial\mathbb{K}_{ij}}{\partial Y^l}\mathbb{K}_{lk}+\frac{\partial\mathbb{K}_{ki}}{\partial Y^l}\mathbb{K}_{lj}\right).
\end{equation}

Since matrix $\mathbb{N}$ and $\mathbb{W}^{-1}$ are block diagonal, we only need to
consider the case in which $i,j,k,l$ are the subscripts for the same particle.
According to $(D_F)_{ij}=\frac{\partial x_i}{\partial y_j}$ and the definition, it has
$N_{ij}=\frac{\partial y_j}{\partial x_i}$.
Since weighting matrix is constant and composed of the reciprocals of particles' weights.
Thus, we can leave out them when verifying the identity, and term (\ref{braN2}) turns into
\begin{align*}
&\frac{\partial}{\partial Y^l}\frac{\partial Y_{i}}{\partial X^j}\frac{\partial Y_{l}}{\partial X^k}-\frac{\partial}{\partial Y^l}\frac{\partial Y_{i}}{\partial X^k}\frac{\partial Y_{l}}{\partial X^j}\\
=&\frac{\partial^2 Y^i}{\partial X^k\partial X^j}-\frac{\partial^2 Y^i}{\partial X^j\partial X^k}\\
=&0.
\end{align*}
If $j\in\{1,\cdots,3N_p\},i\in\{3N_p+1,\cdots,6N_p\}$, it is similar to this case and will not be repeated.

{\bf{Case 3.}}  When $i,j\in\{3N_p+1,\cdots,6N_p\}$, it has $N^\top_{ij}=\frac{\partial y_i}{\partial x_j}$ and
\begin{equation*}
\hat{\bm{B}}=N\hat{\tilde{\bm{B}}}N^\top,
\end{equation*}
Thus the following proof is similar to the case in~\cite{GAJ2022Hamiltonian}.

Overall, if the magnetic field is divergence-free, Jacobi identity of bracket~(\ref{CCbra}) is satisfied.

\end{proof}

\section{Asymptotic behavior of the MVP equation}

The asymptotic behavior of the MVP equation under strong magnetic fields varies case by case, the  tailored analysis to specific magnetic field configurations is necessary.
In the rest of this paper, we focus on simulating the asymptotic behavior of the two-dimensional MVP equation with a fixed direction magnetic field $B(x)=[0,0,b(x_1,x_2)]^\top/\varepsilon$.
And in order to study such equation, one should take a small parameter related to the magnetic field and the time scale which leads to the following system
\begin{equation}\label{2dCVP}
\left\{
\begin{aligned}
&\varepsilon\frac{\partial \tilde{f}}{\partial t}+N(y)^{T}v\cdot\frac{\partial \tilde{f}}{\partial y}+\left(N(y)\tilde{\bm{E}}+\frac{\tilde{b}(y)K}{\varepsilon}v\right)\cdot\frac{\partial \tilde{f}}{\partial v}=0,\\
&-N(y)\nabla_y\cdot N(y)\nabla_y \tilde{\phi}=\int_{\Omega_v} \tilde{f}dv-\rho_0.
\end{aligned}
\right.
\end{equation}
where $K=\left(
\begin{array}{cc}
	0 & 1 \\
	-1 & 0
\end{array}
\right).$
In this case, $y\in\tilde{\Omega}_y\subset\mathbb{R}^{2},v\in\Omega_{v}\subset\mathbb{R}^{2}$ and it has a strong magnetic field, which is $\bm{B}(x)=\frac{1}{\varepsilon}[0,0,b(x)]^\top$ and $\tilde{b}(y):=b(F(y))$.

\subsection{Asymptotic behavior of \eqref{2dCVP} in maximal ordering case}
First of all, we discuss the asymptotic behavior of the characteristic line equation of the MVP equation \eqref{2dCVP}. 
In the maximal ordering case, the  characteristic line equation of the MVP equation reads
\begin{equation}\label{CPD}
\left\{
\begin{aligned}
&\varepsilon\dot{y}=N^\top(y)v,\\
&\varepsilon\dot{v}=N(y)\tilde{\bm{E}}(y)+\frac{\tilde{b}(y)K}{\varepsilon}v,\\
&y(0)=y_0,\ v(0)=v_0,
\end{aligned}
\right.
\end{equation}
where $\tilde{b}(y)=\bar{b}(\varepsilon y)$ and $\tilde{b}(0)=b_0\neq 0$.
Then we can get the following theorem\footnote{The results are corresponding to conclusions of Cartesian coordinate in~\cite{Meiss1990Canonical,Filbet2016,Degond2016,Filbet2020Asymptotics}.}.
\begin{theorem}\label{GCODE}
Assume that $\tilde{b}\in C^1(\tilde{\Omega}_y)$ and $\tilde{\Omega}_y$ is compact.
Then, in the limit $\varepsilon\to 0$ of \eqref{CPD}, it follows that $y\to u$, where $u$ corresponds to the guiding center approximation
\begin{equation}\label{GCeq1}
\dot{u}=\frac{K\tilde{\bm{E}}(u)}{b_0\mathbb{J}(y)}=:\mathcal{R}(u),\quad u(0)=y(0).
\end{equation}
\end{theorem}

\begin{proof}
By multiplying the second equation of \eqref{CPD} by $\frac{K}{\tilde{b}(y)}$, one can get
\begin{equation*}
\varepsilon\frac{K\dot{v}}{\tilde{b}(y)}=\frac{KN(y)\tilde{\bm{E}}(y)}{\tilde{b}(y)}-\frac{v}{\varepsilon}
\end{equation*}
Then the solution $y$ of \eqref{CPD} converges uniformly as $\varepsilon\to 0$ to the deterministic trajectory of \eqref{GCeq1} due to
$$N(y)^\top KN(y)=\frac{K}{\mathbb{J}(y)}.$$
\end{proof}

\subsection{Time discretization}

Then, we discuss the time discretization for the semi-discrete systems (\ref{disCVPODE}).
Due to the complexity of matrix $N(y)$, the semi-discrete system can not be split as several subsystems which can be solved explicitly.
Thus, according to the physical problem we develop the other numerical methods and apply them to simulate the dynamical behavior of charged particles in plasma physics~\cite{Hairer2017Symmetric,Hairer2018Energy,Hairer2020Long}.
In practical magnetic confinement devices, the external magnetic field is very strong (i.e. $0<\varepsilon\ll 1$), which brings a new time scale (Larmor gyration) to the original particle models.
When using the classical numerical methods to simulate these problems, the dynamic behavior of charged particles can only be accurately characterized with step size less than the cyclotron period.
The numerical simulation will suffer form expansive computation over long-time.
Recently, a series of works~\cite{Chartier2020,Hairer2020A,Hairer2022Large} have been devoted to constructing numerical schemes in breaking the time step limit in Cartesian coordinate.
On the other hand, asymptotic-preserving (AP) schemes~\cite{JS1999E,Jin2022Asymptotic} provide a generic framework for such multiscale problems.
Based on this idea, in this section we will construct the suitable numerical methods for systems in Curvilinear coordinate.

To solve system (\ref{CPD}), we apply the implicit-explicit Runge-Kutta schemes developed in~\cite{Pareschi2005Implicit}.
In this paper, we focus on the following two types.

{\bf APSI1:} By denoting $\tau=\frac{\Delta t}{\varepsilon}$ and $\lambda=\frac{\Delta t}{\varepsilon^2}$, a first order semi-implicit scheme is given by
\begin{equation}\label{ASI}
\left\{
\begin{aligned}
&y^{n+1}=y^n+\tau N(y^n)^{T} v^{n+1},\\
&v^{n+1}=v^n+\tau N(y^n)\tilde{\bm{E}}(y^n)+\lambda b(F(y^n))K v^{n+1}.
\end{aligned}
\right.
\end{equation}
(\ref{ASI}) equals to
\begin{equation*}
\left(
\begin{array}{c}
y^{n+1}\\
v^{n+1}\\
\end{array}
\right)=
\left(
\begin{array}{cc}
I & -\tau N(y^n)^{T}\\
0 & I-\lambda b(F(y^n))K\\
\end{array}
\right)^{-1}
\left(
\begin{array}{c}
y^n\\
v^n+\tau N(y^n)\tilde{\bm{E}}(y^n)\\
\end{array}
\right).
\end{equation*}
For the upper triangular matrix above, its inverse matrix is easy to calculate according to covariant basis and skew-symmetric matrix $K$.
Specifically, it is
\begin{equation*}
\left(
\begin{array}{cc}
I & \tau N(y^n)^{T}+\frac{\tau\lambda b(F(y^n))}{1+(\lambda b(F(y^n)))^2}N(y^n)^{T}K+\frac{\tau\lambda^2 b(F(y^n))^2}{1+(\lambda b(F(y^n)))^2}N(y^n)^{T}K^2\\
0 & I+\frac{\lambda b(F(y^n))}{1+(\lambda b(F(y^n)))^2}K+\frac{\lambda^2 b(F(y^n))^2}{1+(\lambda b(F(y^n)))^2}K^2
\end{array}
\right).
\end{equation*}

{\bf APSI2:} Moreover, a second order L-stable scheme can be written as
\begin{equation}\label{ASI2}
\left\{
\begin{aligned}
&y_n^{(1)}=y^n+\gamma\tau N(y^n)^{T} v_n^{(1)},\\
&v_n^{(1)}=v^n+\gamma\tau F_n^{(1)},\\
&y^{n+1}=y^n+(1-\gamma)\tau N(y^n)^{T} v_n^{(1)}+\gamma\tau N(y^n)^{T} v^{n+1},\\
&v^{n+1}=v^n+(1-\gamma)\tau F_n^{(1)}+\gamma\tau F_n^{(2)}.
\end{aligned}
\right.
\end{equation}
where $\gamma=1-1/\sqrt{2}$ and
\begin{align*}
&y_n^{(2)}=y^n+\frac{\tau}{2\gamma}N(y^n)^{T}v_n^{(1)},\\
&F_n^{(1)}=N(y^n)\tilde{\bm{E}}(y^n)+\frac{b(F(y^n))}{\varepsilon}Kv_n^{(1)},\\
&F_n^{(2)}=N(y_n^{(2)})\tilde{\bm{E}}(y_n^{(2)})+\frac{b(F(y_n^{(2)}))}{\varepsilon}K v^{n+1}.
\end{align*}
For a vector form, it reads
\begin{equation*}
\begin{aligned}
&\left(
\begin{array}{c}
y_n^{(1)}\\
v_n^{(1)}\\
\end{array}
\right)=
\left(
\begin{array}{cc}
I & -\gamma\tau N(y^n)^{T}\\
0 & I-\gamma\lambda b(F(y^n))K\\
\end{array}
\right)^{-1}
\left(
\begin{array}{c}
y^n\\
v^n+\gamma\tau N(y^n)\tilde{\bm{E}}(y^n)\\
\end{array}
\right),\\
&\left(
\begin{array}{c}
y^{n+1}\\
v^{n+1}\\
\end{array}
\right)=
\left(
\begin{array}{cc}
I & -\gamma\tau N(y_n^{(2)})^{T}\\
0 & I-\gamma\lambda b(F(y_n^{(2)}))K\\
\end{array}
\right)^{-1}
\left(
\begin{array}{c}
y^n+(1-\gamma)\tau N(y^n)^{T}v_n^{(1)}\\
v^n+(1-\gamma)\tau F_n^{(1)}+\gamma\tau N(y_n^{(2)})\tilde{\bm{E}}(y_n^{(2)})\\
\end{array}
\right).
\end{aligned}
\end{equation*}

\subsubsection{Uniform estimates of \eqref{ASI}}

Then we focus on the asymptotic-preserving property of these algorithms.
For convenience, we will use $b^n$ to denote $b(F(y^n))$.
As a result, we have the following theorem for scheme (\ref{ASI}).
\begin{theorem}\label{APTH}
Assume that $\bm{E}\in W^{1,\infty}$, $\tilde{\Omega}_y$ and $\Omega_x$ are compact.
Then for the solution $(y^n,v^n)$ to equation (\ref{ASI}) in a finite time $T$, there exists $\lambda_0$ such that when $\lambda\ge\lambda_0$ the following estimate holds
\begin{equation}\label{Estimate1}
\Vert y^n-u^n \Vert\lesssim\varepsilon^2\left(1+\Vert\varepsilon^{-1}v^0-(b^{0})^{-1}KN(y^0)\tilde{\bm{E}}(y^{0})\Vert\right)
\end{equation}
where $u^n$ is the numerical solution of the following guiding-center model
\begin{equation}\label{CUdis}
u^{n+1}=u^{n}+\Delta t \mathcal{R}(u^{n}),\quad u^0=y^0.
\end{equation}
\end{theorem}

\begin{proof}
Since $\bm{E}\in W^{1,\infty}$, $\tilde{\Omega}_y$ and $\Omega_x$ are compact, electric field, inverse Jacobi matrix, Jacobian with its reciprocal and first-order derivative are all bounded.
So a common upper bound $\kappa$ is taken for estimates.
Firstly we introducing intermediate variables
\begin{equation}\label{znmid}
z^n=\varepsilon^{-1}v^n-(b^{n-1})^{-1}KN(y^{n-1})\tilde{\bm{E}}(y^{n-1}).
\end{equation}
On the other hand, we have
$$v^1=\left(I-\lambda b^0 K\right)^{-1}(v^0+\tau N(y^0)\tilde{\bm{E}}(y^0)).$$
Since $\left(I-\beta K\right)^{-1}=\frac{I+\beta K}{1+\beta^2}$ holds for $\forall \beta\in\mathbb{R}$, it yields that
$$z^1=\frac{I+\lambda b^0 K}{1+(\lambda b^0)^2}\left(\varepsilon^{-1}v^0-(b^{0})^{-1}KN(y^0)\tilde{\bm{E}}(y^{0})\right).$$
And the following equality holds according to (\ref{znmid})
\begin{equation}\label{znplus}
z^{n+1}=\frac{I+\lambda b^n K}{1+(\lambda b^n)^2}\left(z^n-K\left((b^{n})^{-1}N(y^n)\tilde{\bm{E}}(y^{n})-(b^{n-1})^{-1}N(y^{n-1})\tilde{\bm{E}}(y^{n-1})\right)\right)
\end{equation}
for $\forall n\ge 1$.
Due to (\ref{ASI}), it has
\begin{equation}\label{Cdis2}
y^n=y^{n-1}+\tau N(y^{n-1})^\top v^n=y^{n-1}+\Delta t N(y^{n-1})^\top z^n+\Delta t \mathcal{R}(y^{n-1}).
\end{equation}
Let $\Lambda=\left\Vert\frac{1+\lambda b(F(y))}{1+(\lambda b(F(y)))^2}\right\Vert_{\infty}, a:=\Lambda\max\left\{1+2\kappa^2\Delta t,\ 1+2\kappa^3\Delta t\right\}$, $b:=2\Lambda \kappa^4\Delta t$ and $c:=\Lambda\Vert\varepsilon^{-1}v^0-(b^{0})^{-1}KN(y^0)\tilde{\bm{E}}(y^{0})\Vert.$
By fixing the time step $\Delta t$, when $\varepsilon\to 0$, it has $\lambda\to\infty$.
Thus there exists constants $\lambda_0,\ \alpha>0$ such that $a\leq \alpha<1$ when $\lambda\ge\lambda_0$.
After taking the norm on (\ref{znplus})'s both sides, it yields $\Vert z^{n+1}\Vert\leq a^n c+\frac{b}{1-a}.$
By denoting $e^n:=y^n-u^n$, it can be obtained that
\begin{equation*}
\Vert e^n \Vert\leq \frac{a}{\Lambda}\Vert e^{n-1} \Vert+\kappa\Delta t a^{n-1}c+\kappa\Delta t\frac{b}{1-a}.
\end{equation*}
Finally, (\ref{Estimate1}) holds due to $e^0=0$.

\end{proof}

On the other hand, one can get a better estimate if the initial value in (\ref{CUdis}) is modified.
\begin{corollary}
If the conditions in theorem (\ref{APTH}) hold, it has the following estimate:
$$\Vert y^n-u^n \Vert\lesssim\varepsilon^2\left(1+\left(\frac{1}{\lambda}+\Delta t\right)\Vert\varepsilon^{-1}v^0-(b^{0})^{-1}KN(y^0)\tilde{\bm{E}}(y^{0})\Vert\right),$$
where $u^n$ is the numerical solution of equation (\ref{CUdis}) with modified initial value
\begin{equation*}
u^0=y^0+\frac{\varepsilon}{b^{0}} N(y^0)^\top\left(Kv^0+\frac{\varepsilon}{b^{0}} N(y^0)\tilde{\bm{E}}(y^{0})\right).
\end{equation*}
\end{corollary}

\begin{proof}
By noting that
$$z^1=\left(\frac{K}{\lambda b^0}+\frac{\lambda b^0I-K}{\lambda b^0(1+(\lambda b^0)^2)}     \right)\left(\varepsilon^{-1}v^0-(b^{0})^{-1}KN(y^0)\tilde{\bm{E}}(y^{0})\right),$$
and then it has
\begin{equation*}
e^1=\Delta t \mathcal{R}(y^{0})-\Delta t \mathcal{R}(u^{0})+\Delta t N(y^0)^\top\frac{\lambda b^0I-K}{\lambda b^0(1+(\lambda b^0)^2)}\left(\varepsilon^{-1}v^0-(b^{0})^{-1}KN(y^0)\tilde{\bm{E}}(y^{0})\right).
\end{equation*}
Moreover, it follows that
\begin{equation*}
\begin{aligned}
\Vert e^n \Vert&\leq \left(\frac{a}{\Lambda}\right)^{n-1}\Vert e^1 \Vert+C(\kappa)\frac{\Delta t}{\lambda} a^{n-1}c\left(\left(\frac{1}{\Lambda}\right)^{n-1}-1\right)+C(T,\kappa)\frac{b}{1-a},\\
&\leq C(T,\kappa,\alpha,b_0)\varepsilon^{2}(\Delta t+\frac{1}{\lambda})\Vert\varepsilon^{-1}v^0-(b^{0})^{-1}KN(y^0)\tilde{\bm{E}}(y^{0})\Vert+C(T,\kappa)\varepsilon^{2}.
\end{aligned}
\end{equation*}

\end{proof}

\begin{corollary}\label{orederTH}
Under the assumptions of Theorem~\ref{APTH}, the proposed scheme \eqref{ASI} is of order $1$ when $\varepsilon\ll\Delta{t}$. 
More precisely, it has
\begin{equation}\label{Estimateoreder}
\Vert y^N-y(T)\Vert\lesssim \Delta{t},
\end{equation}
for $T=N\Delta{t}$.
\end{corollary}

\begin{proof}
	
By the triangle inequality, we can obtain that
$$\Vert y^N-y(T)\Vert\le\Vert y^N-u^N\Vert+\Vert u^N-u(T)\Vert+\Vert u(T)-y(T)\Vert.$$
According to Theorem~\ref{GCODE} and Theorem~\ref{APTH},  we have
$$\Vert y^N-u^N\Vert+\Vert u(T)-y(T)\Vert\lesssim \varepsilon.$$
Then \eqref{Estimateoreder} holds due to \eqref{CUdis}.
	
\end{proof}

\subsubsection{Uniform estimates of \eqref{ASI2}}
For scheme (\ref{ASI2}), it follows:
\begin{theorem}\label{APTH2}
Under the assumptions of theorem (\ref{APTH}), the estimate (\ref{Estimate1}) holds for $(y^n,v^n)$ to equation (\ref{ASI2}) while $u^n$ is the numerical solution of the guiding-center model
\begin{equation}\label{CUdis2}
u^{n+1}=u^{n}+(1-\gamma)\Delta t \mathcal{R}(u^{n})+\gamma\Delta t \mathcal{R}(u^{(1)}),\quad u^{(1)}=u^{n}+\frac{\Delta t}{2\gamma}\mathcal{R}(u^{n}),\quad u^0=y^0.
\end{equation}
\end{theorem}

\begin{proof}
The proof is similar to Theorem~\ref{APTH}, we briefly present the key points.
By using $\mathcal{F}(y)=D_F(y)\mathcal{R}(y)$, we can define
\begin{equation*}
z^n_{1}=\varepsilon^{-1}v_n^{(1)}-\mathcal{F}(y^{n-1}),\quad z^n_{2}=\varepsilon^{-1}v_n-\mathcal{F}(y_{n-1}^{(2)}).
\end{equation*}
Then due to scheme (\ref{ASI2}), it follows that
\begin{align*}
&z^{n+1}_{1}=\left(I-\gamma\lambda b^n K\right)^{-1}(z_2^n+\mathcal{F}(y_{n-1}^{(2)})-\mathcal{F}(y_{n})),\\
&z^{n+1}_{2}=\left(I-\gamma\lambda b(F(y_n^{(2)})) K\right)^{-1}(z_2^n+\mathcal{F}(y_{n-1}^{(2)}))-\mathcal{F}(y_{n}^{(2)}+(1-\gamma)\lambda b^n Kz_1^{n+1}).
\end{align*}
By denoting $\Lambda=\left\Vert\frac{1+\lambda\gamma b(F(y))}{1+(\lambda\gamma b(F(y)))^2}\right\Vert_{\infty}$, it has
\begin{equation}\label{errz12}
\begin{aligned}
&\Vert z^{n+1}_{1}\Vert\le C(\kappa)\Lambda(\Vert z^{n}_{2}\Vert+\Delta t\Vert z^{n}_{1}\Vert+\Delta t(1+\gamma)),\\
&\Vert z^{n+1}_{2}\Vert\le C(\kappa)\Lambda\left(\left(\frac{\Delta t}{2\gamma}+(1-\gamma)\lambda\right)\Vert z^{n+1}_{1}\Vert+\Delta t\Vert z^{n}_{1}\Vert+\Vert z^{n}_{2}\Vert+\Delta t\right),
\end{aligned}
\end{equation}
due to
\begin{align*}
&\Vert y^{(2)}_{n-1}-y^n\Vert\le C(\kappa)\tau\Vert v_{n-1}^{(1)}-\gamma v^n \Vert \le C(\kappa)\Delta t(\Vert z^{n}_{1}\Vert+\Vert z^{n}_{2}\Vert+(1+\gamma)),\\
&\Vert y^{(2)}_{n-1}-y^{(2)}_{n}\Vert\le C(\kappa)\Delta t\left(\frac{1}{2\gamma}\Vert z^{n+1}_{1}\Vert+\Vert z^{n}_{1}\Vert+\gamma\Vert z^{n}_{2}\Vert+1\right).
\end{align*}
Then (\ref{errz12}) leads to
\begin{equation*}
\Vert z^{n+1}_{1}\Vert+\Vert z^{n+1}_{2}\Vert\le C(\kappa)\Lambda(\Vert z^{n}_{1}\Vert+\Vert z^{n}_{2}\Vert+\Delta t).
\end{equation*}
On the other hand, by taking difference between $y^{n+1}$ and $u^{n+1}$, (\ref{Estimate1}) holds due to
\begin{equation*}
\Vert e^{n+1}\Vert\le C(\kappa)(\Vert e^{n}\Vert+\Delta t\Vert z^{n+1}_{1}\Vert+\Delta t\Vert z^{n+1}_{2}\Vert)
\end{equation*}
and
\begin{align*}
&z^1_{1}=\frac{I+\gamma\lambda b^0 K}{1+(\gamma\lambda b^0)^2}\left(\varepsilon^{-1}v^0-(b^{0})^{-1}KN(y^0)\tilde{\bm{E}}(y^{0})\right),\\
&z^1_{2}=\frac{I+\gamma\lambda b(F(y_{0}^{(2)})) K}{1+\left(\gamma\lambda b(F(y_{0}^{(2)}\right)^2}\left(\varepsilon^{-1}v^0+(1-\gamma)\lambda b^0 z^1_{1}-\mathcal{F}(y_{0}^{(2)})\right).
\end{align*}

\end{proof}

\begin{corollary}\label{orederTH2}
Under the assumptions of Theorem~\ref{APTH}, the proposed scheme \eqref{ASI2} is of order $2$ when $\varepsilon\ll\Delta{t}^2$. 
More precisely, it has
\begin{equation*}
\Vert y^N-y(T)\Vert\lesssim \Delta{t}^2,
\end{equation*}
for $T=N\Delta{t}$.
\end{corollary}

\section{Numerical experiments}

In this section, we display some examples to verify the analysis results. 
By simulating the motion of a single particle, the asymptotic-preserving properties of schemes can be observed and compare to our theoretical results.
Then an application of a $2+2$-dimensional Vlasov system is presented to display the advantage of our algorithm.
Last but not least, we consider the polar coordinate transformation $x_1=r\cos(\theta),\  x_2=r\sin(\theta)$ for $y=(r,\theta)$.

\subsection{Charged particle dynamics}
First, a benchmark numerical test is carried out to show the efficiency of our proposed algorithms and verify the uniform estimate of them.
We consider the single particle motion with the electromagnetic field  
$$\bm{E}(x)=-x,\ b(x)=1+\varepsilon\sin(\sqrt{x_1^2+x_2^2}),$$
and the initial data 
$$y(0)=(0.36,0.6)^\top,\ v(0)=(-0.7,0.08)^\top.$$

\begin{figure}[h!]
\centering
\subfigure[]{
\includegraphics[scale=.45]{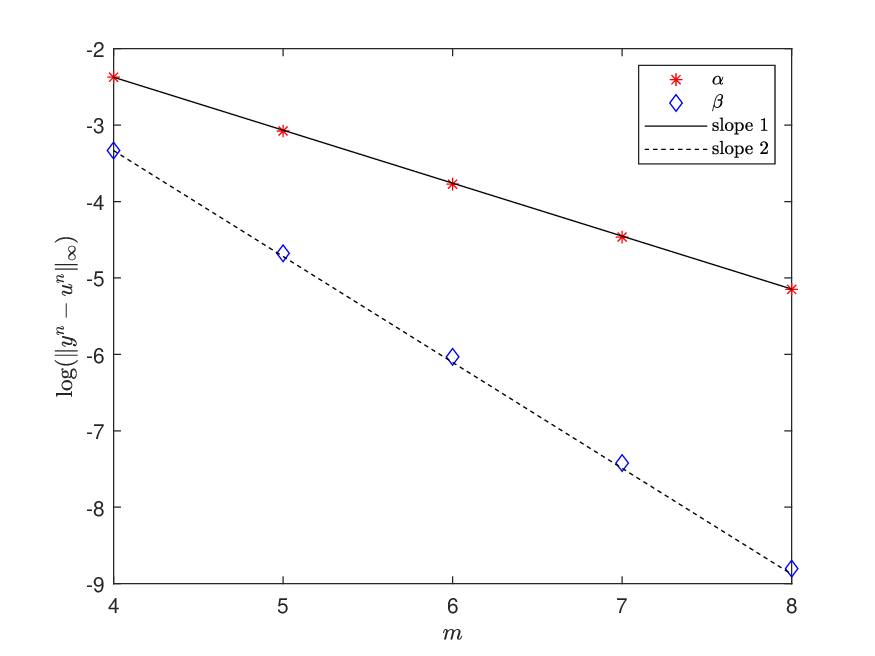}
}
\subfigure[]{
\includegraphics[scale=.45]{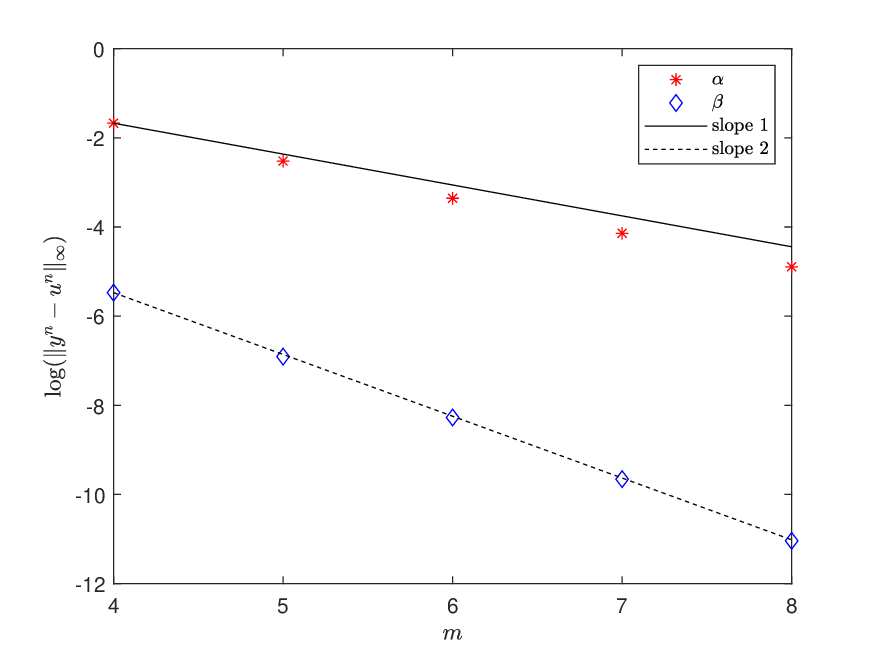}
}
\caption{Uniform estimates of semi-implicit schemes with time step $\Delta t=0.1$ for $\varepsilon=2^{-m}$ and $T=10$. $\alpha$: Original initial data. 
$\beta$: Modified initial data $v^0=\varepsilon(b^{0})^{-1}KN(y^0)\tilde{\bm{E}}(y^{0})$. (a): APSI1 (\ref{ASI}). (b): APSI2 (\ref{ASI2}).}
\label{img1}
\end{figure}

In Figure~\ref{img1}, it can be observed that numerical errors match our theoretical estimations well (by noticing $\Vert\varepsilon^{-1}v^0-(b^{0})^{-1}KN(y^0)\tilde{\bm{E}}(y^{0})\Vert=\mathcal{O}(\varepsilon^{-1})$, it can be obtained that $\Vert e^n\Vert=\mathcal{O}(\varepsilon)$).
After modifying the initial data $v^0$, it leads to $\Vert e^n\Vert=\mathcal{O}(\varepsilon^{2})$.
We validated the uniform accuracy of the two numerical algorithms with respect to $\varepsilon$ in Figures~\ref{img4} and~\ref{img5}, respectively.
The reference solutions were obtained using the Boris algorithm by small time steps ($\Delta{t}=\mathcal{O}(\varepsilon^2)$). 
From the figures, it can be observed that when $\varepsilon$ is sufficiently small, the errors of our algorithms are independent of $\varepsilon$ and achieve the theoretical order of convergence.

\begin{figure}[h!]
\centering
\subfigure[]{
\includegraphics[scale=.45]{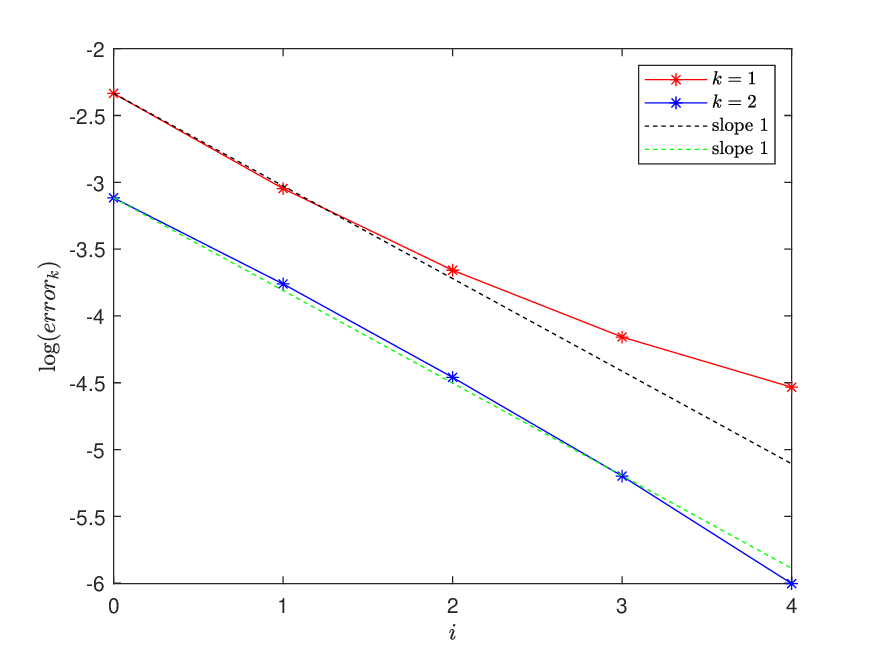}
}
\subfigure[]{
\includegraphics[scale=.45]{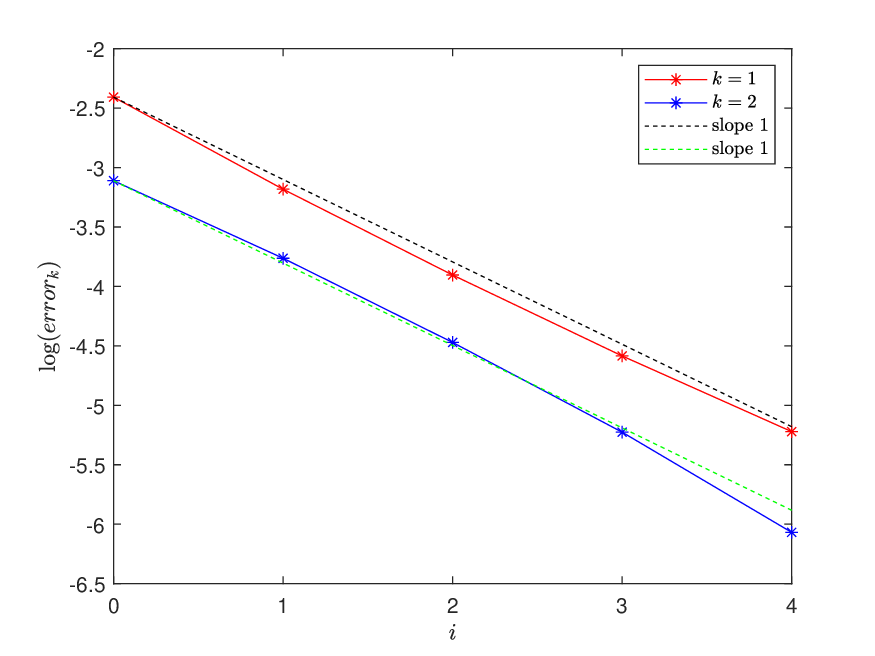}
}
\caption{Estimates $error_k=\vert y_k^N-y_k(T)\vert$ of semi-implicit scheme \eqref{ASI} with time step $\Delta t=\frac{\pi}{20}2^{-i}$ at final time $T=\pi$. (a): $\varepsilon=10^{-2}$. (b): $\varepsilon=10^{-3}$.}
\label{img4}
\end{figure}

\begin{figure}[h!]
\centering
\subfigure[]{
\includegraphics[scale=.45]{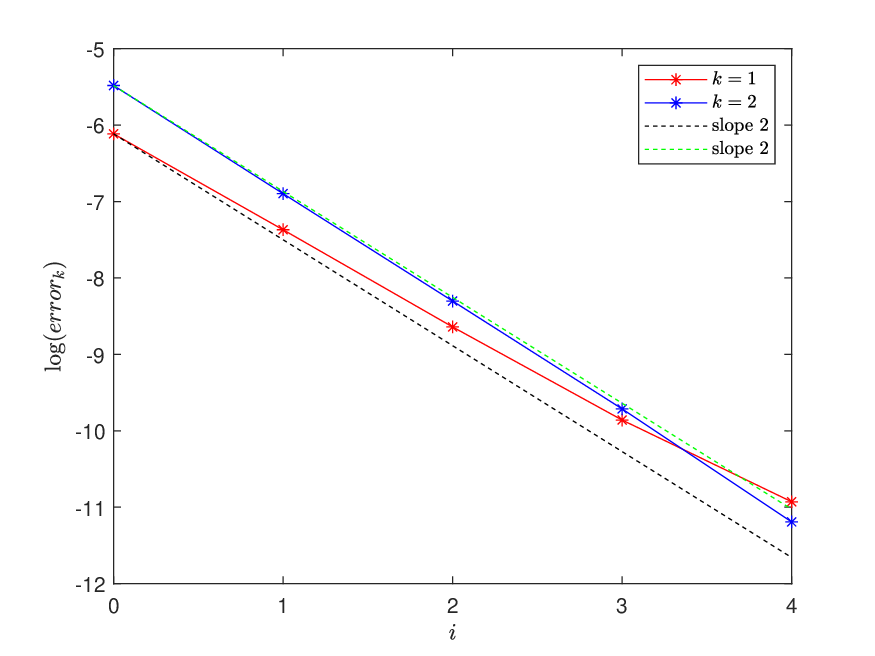}
}
\subfigure[]{
\includegraphics[scale=.45]{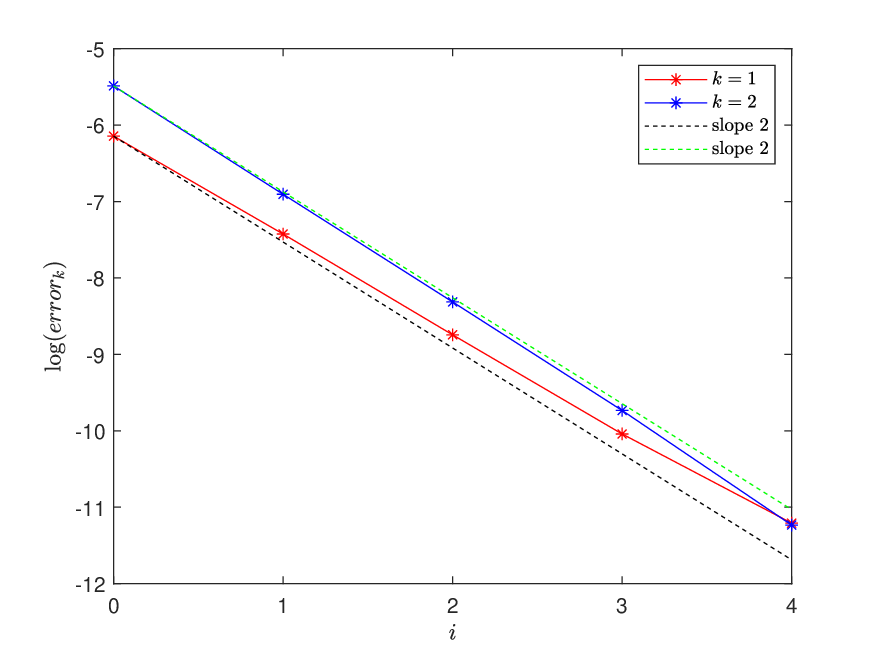}
}
\caption{Estimates $error_k=\vert y_k^N-y_k(T)\vert$ of semi-implicit scheme \eqref{ASI2} with time step $\Delta t=\frac{\pi}{5}2^{-i}$ at final time $T=\pi$. (a): $\varepsilon=10^{-3}$. (b): $\varepsilon=10^{-4}$.}
\label{img5}
\end{figure}

\begin{figure}[h!]
\centering
\subfigure[]{
\includegraphics[scale=.33]{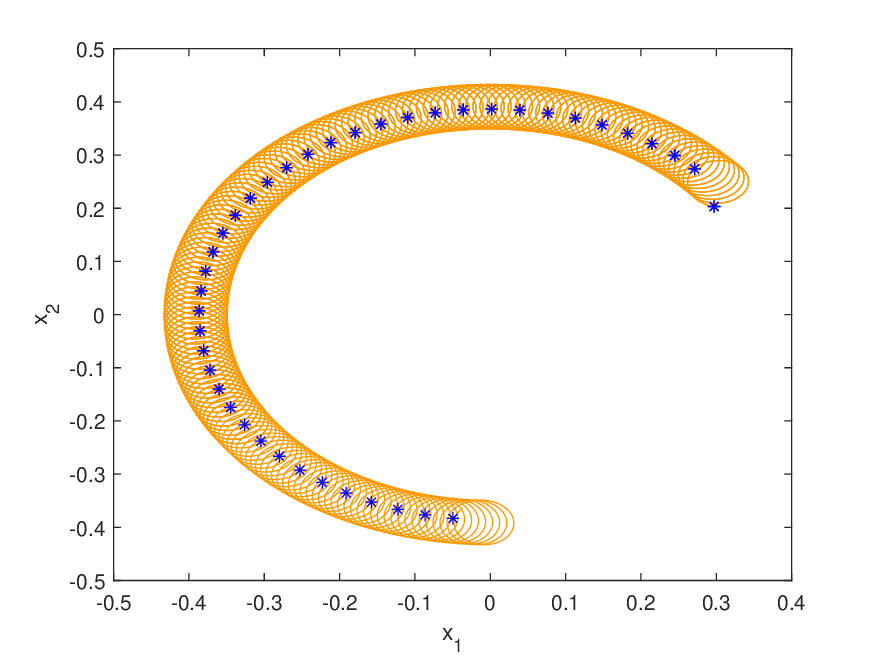}
}
\subfigure[]{
\includegraphics[scale=.33]{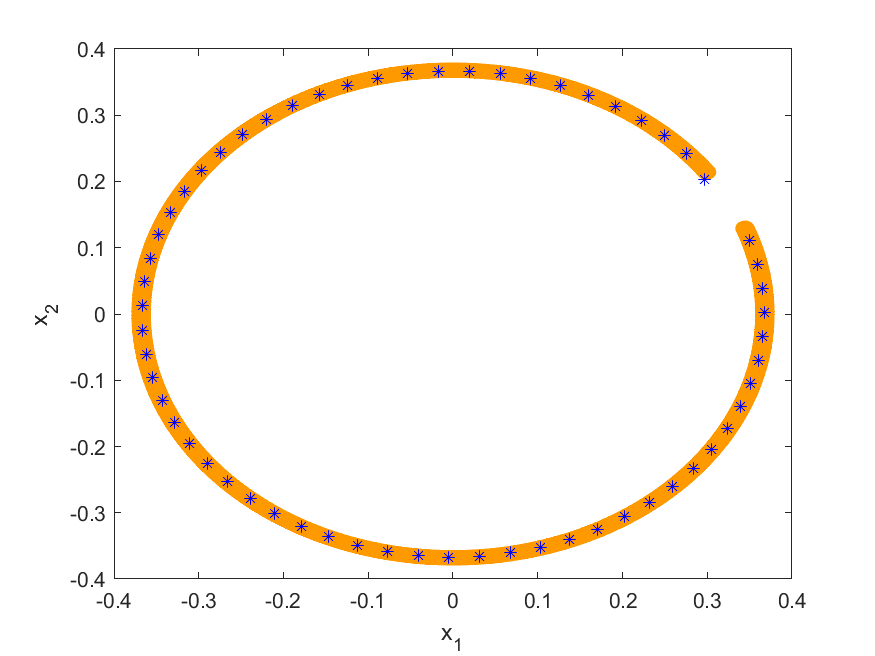}
}
\subfigure[]{
\includegraphics[scale=.3]{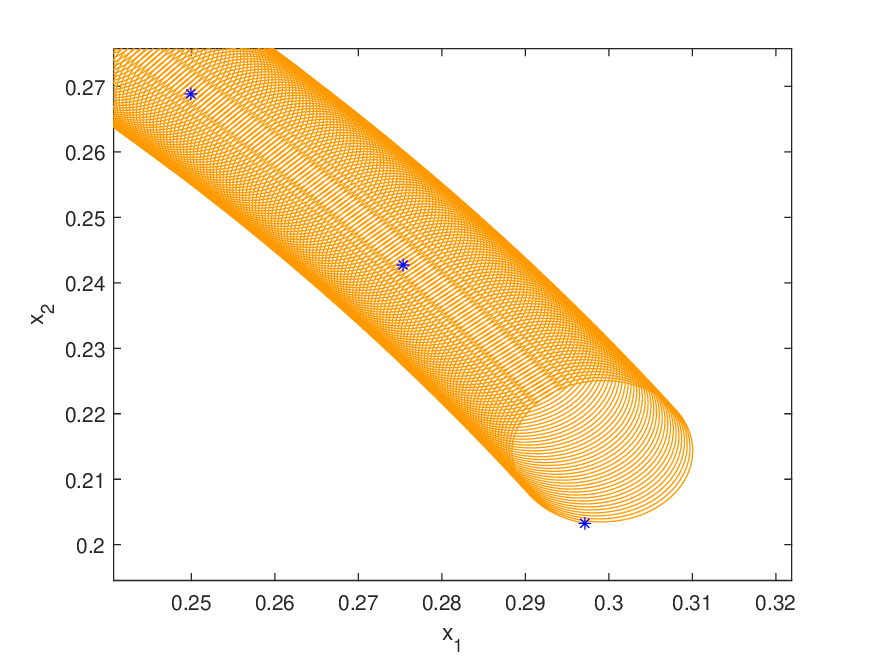}
}
\caption{Comparison between a reference solution (Boris method with time step $\varepsilon^2$) and the solution of algorithm (\ref{ASI}) with time step $\Delta t=0.1$. (a): $\varepsilon=2^{-4},T=1/\sqrt{\varepsilon}$. (b): $\varepsilon=2^{-6},T=0.75/\sqrt{\varepsilon}$. (c): A zoom of the left part.}
\label{img2}
\end{figure}

\begin{figure}[h!]
\centering
{
\includegraphics[scale=.4]{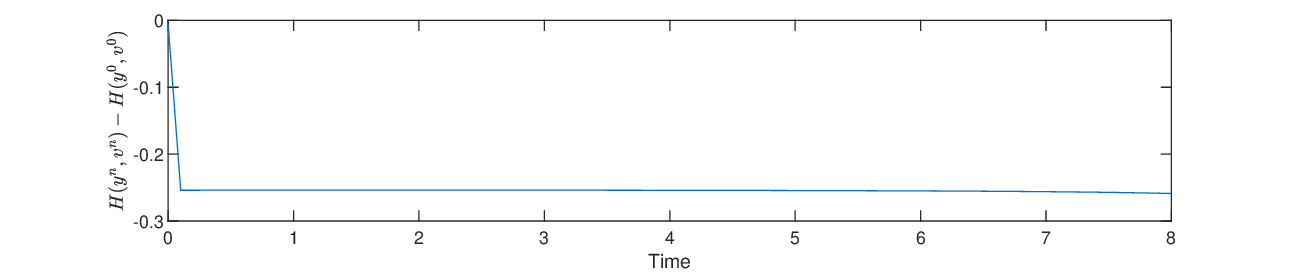}
}
\caption{Time evolutions of error of energy with time step $\Delta t=0.1$ for $\varepsilon=2^{-8}$ and $T=0.5/\sqrt{\varepsilon}$. Here, $\bm{E}(x)=-\nabla\phi(x)$ with $\phi=1/3(x_1^3+x_2^3)$.}
\label{img3}
\end{figure}

Since numerical results of two algorithms in Figure~\ref{img2} and Figure~\ref{img3} are nearly the same, so we just plot one of them.
And they show our method is able to capture the high oscillations of the solution even with a very coarse time step and preserve the energy well over long time.
At this time, orbit calculated by our numerical method parallels to the guiding center line.
In (c) of Figure~\ref{img2}, it can be observed that the particle directly moves to this orbit at the first step.
Thus, the error of energy increases at the beginning in Figure~\ref{img3}.
Thereafter, the energy of the system hardly changes for a long time.
While our temporal discretization does not strictly conserve energy, its asymptotic-preserving property ensures well-controlled error growth.

\subsection{Diocotron instability}
In this experiment, the magnetic field is uniform with $\tilde{b}(y)=1$ and initial distribution function is taken as
\begin{equation*}
\tilde{f}_0(y,v)=\frac{\tilde{\rho}_0(y)}{2\pi}\exp(-\frac{{\vert v\vert}^2}{2}),\quad y=(r,\theta)\in\mathbb{R}_{+}\times\mathbb{R},
\end{equation*}
where the initial density is
\begin{equation*}
\tilde{\rho}_0(\mathbf{x})=
\begin{cases}
(1+{\alpha}\cos(l\theta))\exp(-4(r-6.5)^2)&\mbox{if $r^-\leq r\leq{r}^+$,} \\
0&\mbox{otherwise,}
\end{cases}
\end{equation*}
with $l$ the number of vortices. 
In our simulation, we take $r^-=5,r^+=8,\alpha=0.2$. 
The parameter $\varepsilon$ is used to control the strength of the magnetic field.
We consider the particle equations on the orthogonal planes to magnetic field after transformation, that is
\begin{equation*}
\left\{
\begin{aligned}
&\varepsilon\dot{y}=N^{T}v,\\
&\varepsilon\dot{v}=N\tilde{\bm{E}}_h+\frac{1}{\varepsilon}Kv,\\
&\tilde{\bm{E}}_h=-\nabla_y\tilde{\phi}_h,\\
&\int_{\tilde{\Omega}_y}\left(r\frac{\partial\tilde{\phi}_h}{\partial r}\frac{\partial \psi}{\partial r}+\frac{1}{r}\frac{\partial\tilde{\phi}_h}{\partial \theta}\frac{\partial\psi}{\partial \theta}\right)dy=\int_{\tilde{\Omega}_y}\psi\sum_{s=1}^{N_p}\alpha_s \zeta_r(y-Y_s(t)) dy,\ \forall\psi\in H^1_0(\tilde{\Omega}_y).
\end{aligned}
\right.
\end{equation*}
where
$N=\left(
                \begin{array}{cc}
                  \cos{\theta} & -\frac{\sin{\theta}}{r} \\
                  \sin{\theta} & \frac{\cos{\theta}}{r} \\
                \end{array}
              \right),$ and $\zeta_r$ is a regularizing function for Dirac function.

\begin{figure}[h!]
\centering
\includegraphics[scale=.5]{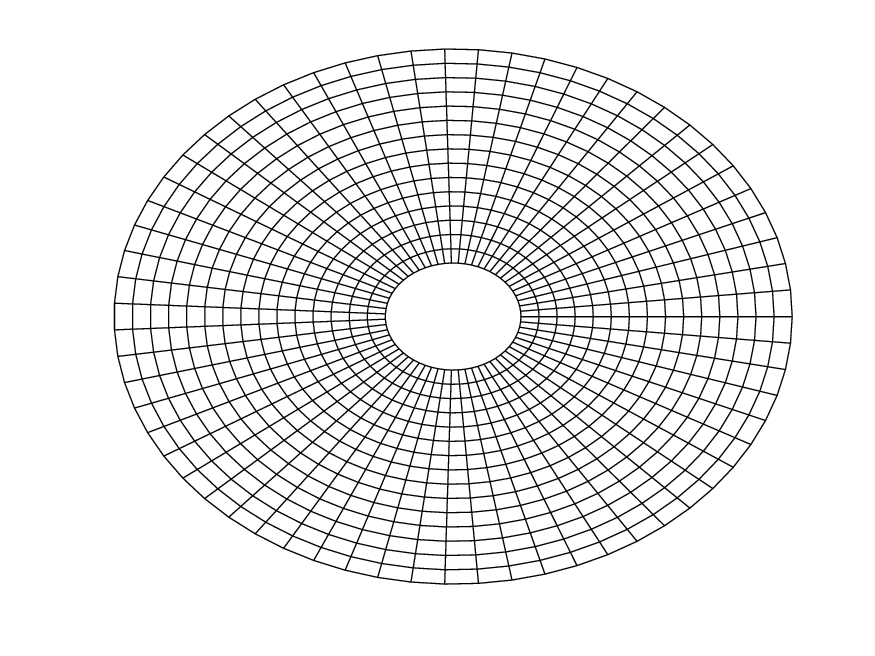}
\caption{Grid of the computational domain.}
\label{imggrid}
\end{figure}

\begin{figure}[!ht]
\centering
\subfigure{
\includegraphics[scale=.45]{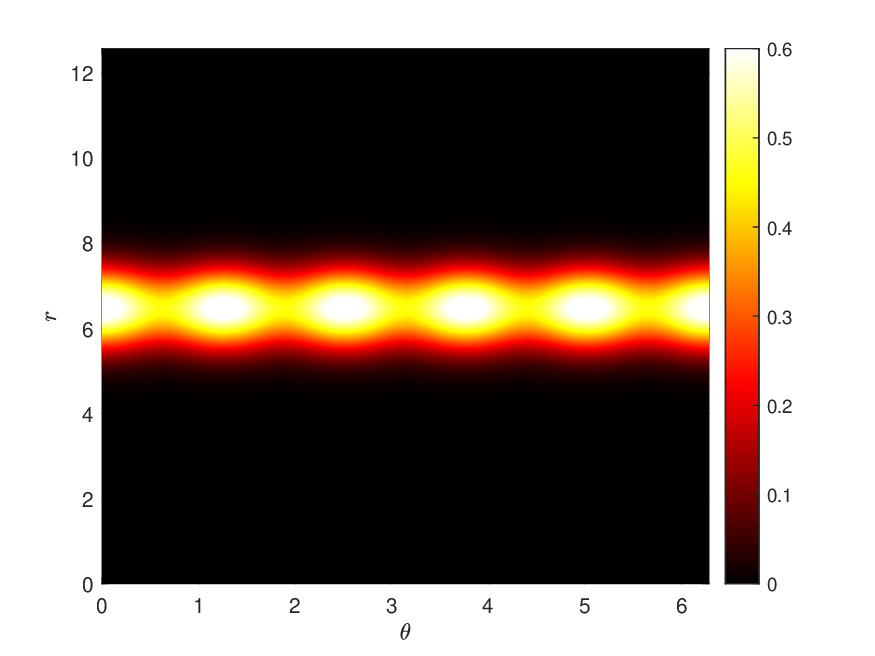}
}
\subfigure{
\includegraphics[scale=.45]{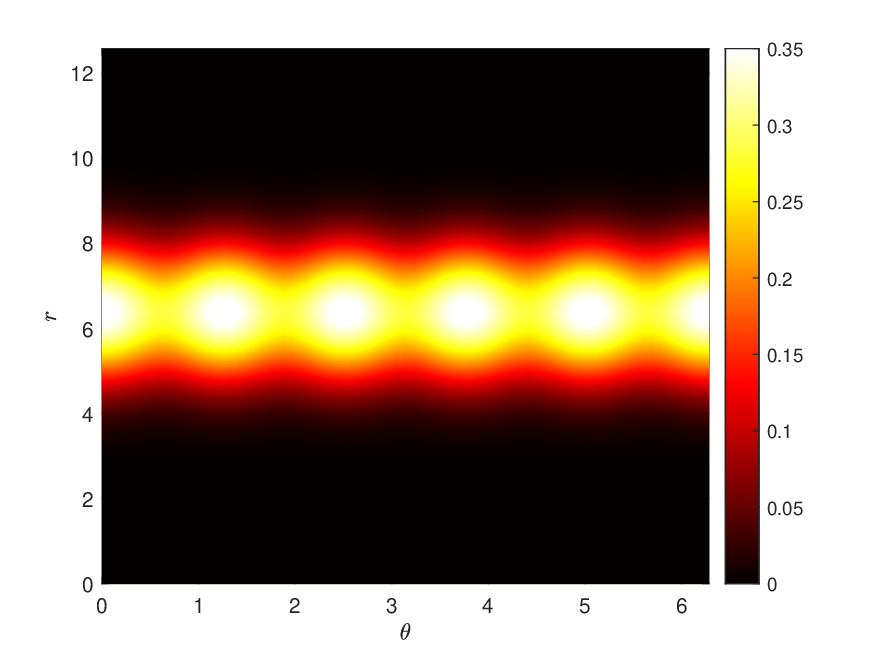}
}
\subfigure{
\includegraphics[scale=.45]{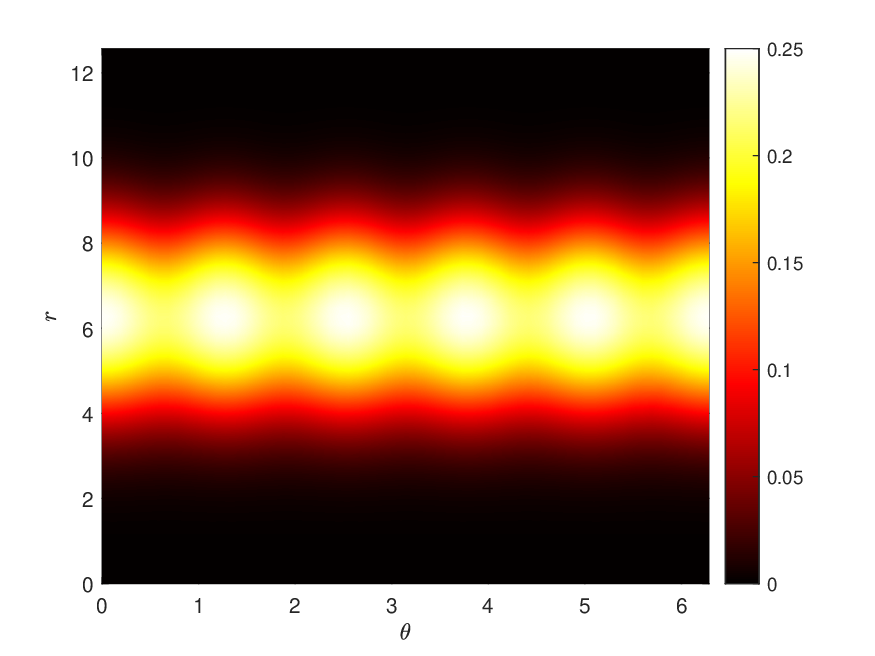}
}
\subfigure{
\includegraphics[scale=.45]{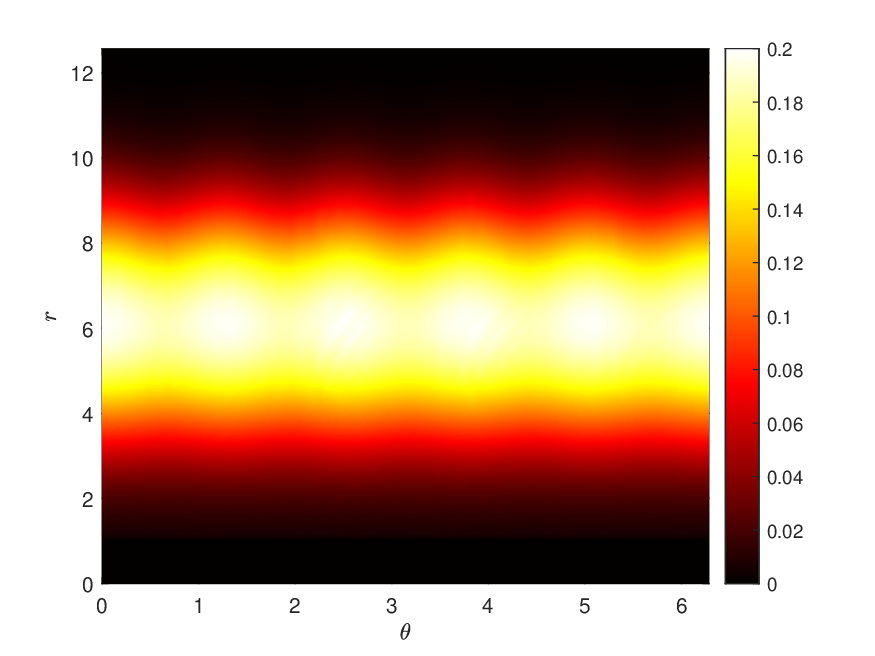}
}
\caption{Time evolution of the density for time $t=0.5,1,1.5,2$ and $l=5$ with parameter $\varepsilon=1$.}
\label{imgCC0}
\end{figure}

\begin{figure}[!ht]
\centering
\subfigure{
\includegraphics[scale=.33]{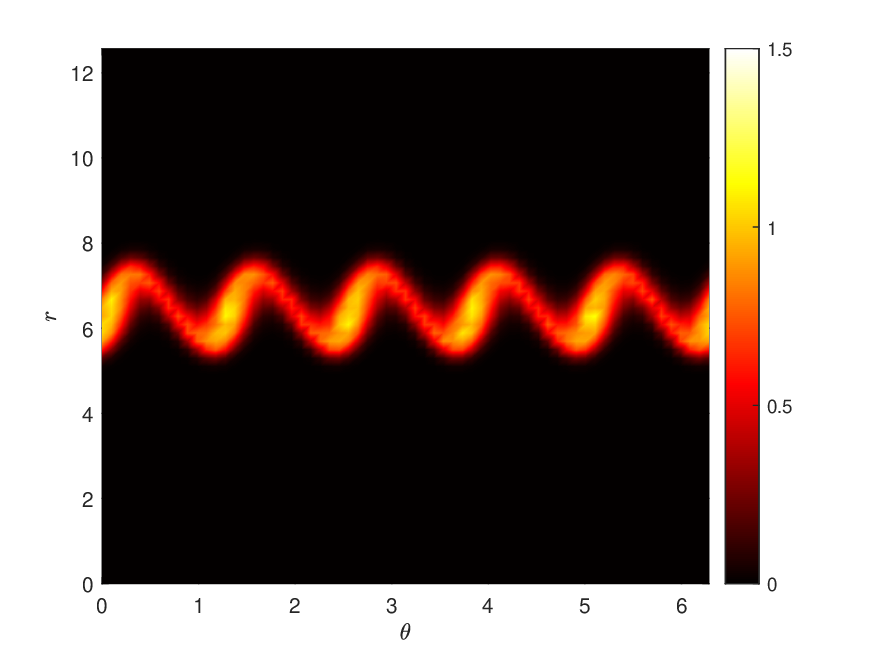}
}
\subfigure{
\includegraphics[scale=.33]{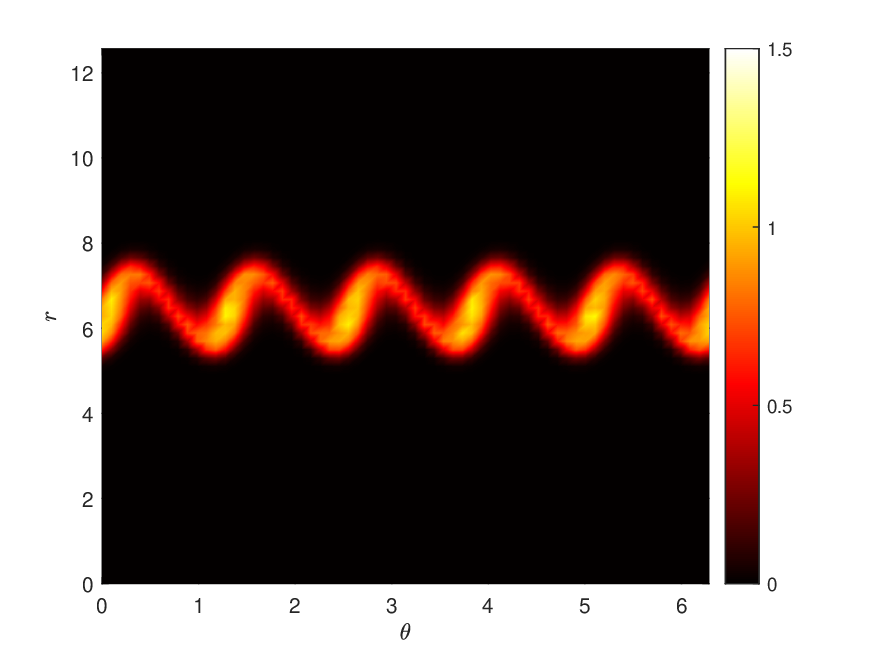}
}
\subfigure{
\includegraphics[scale=.33]{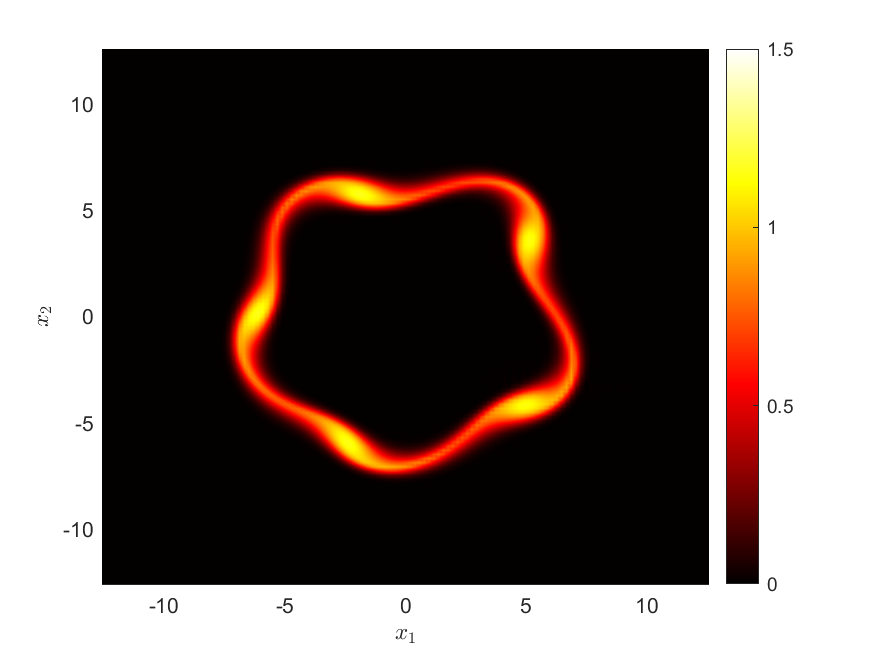}
}
\subfigure{
\includegraphics[scale=.33]{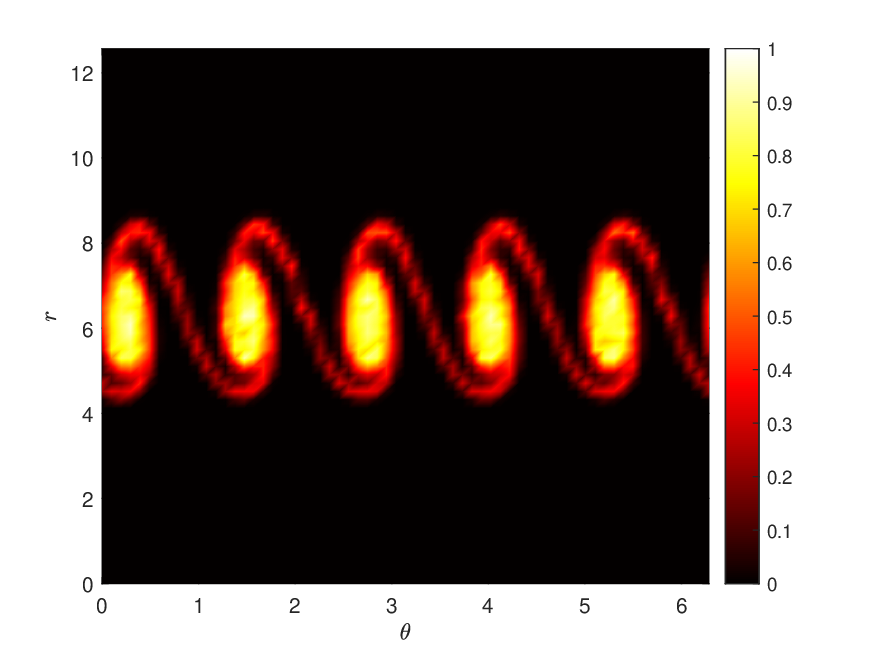}
}
\subfigure{
\includegraphics[scale=.33]{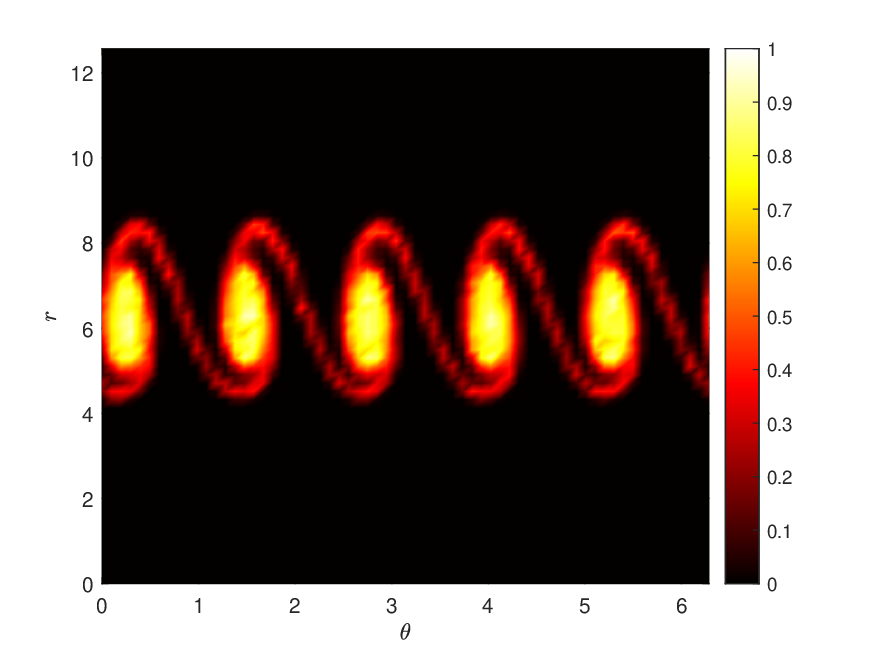}
}
\subfigure{
\includegraphics[scale=.33]{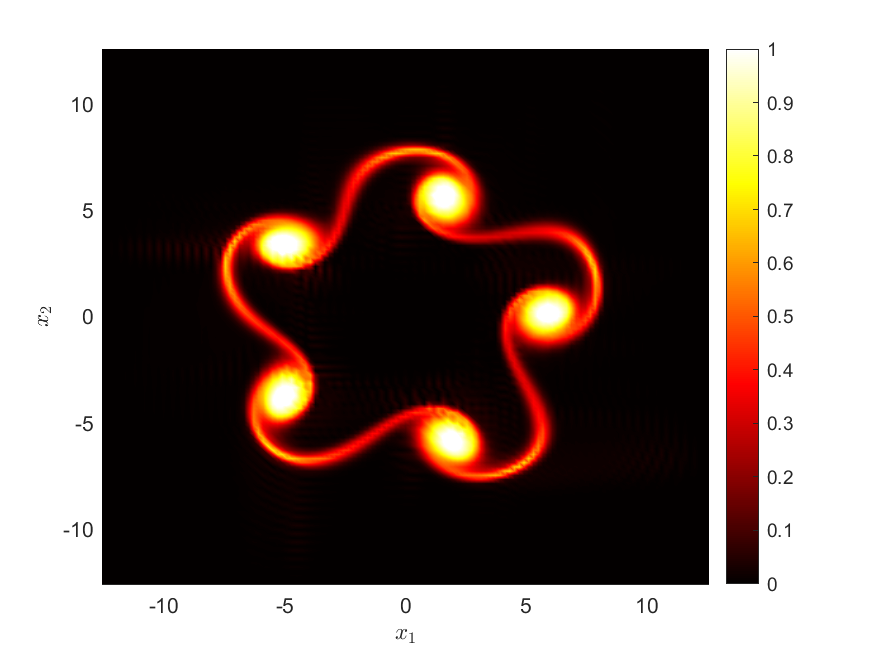}
}
\subfigure{
\includegraphics[scale=.33]{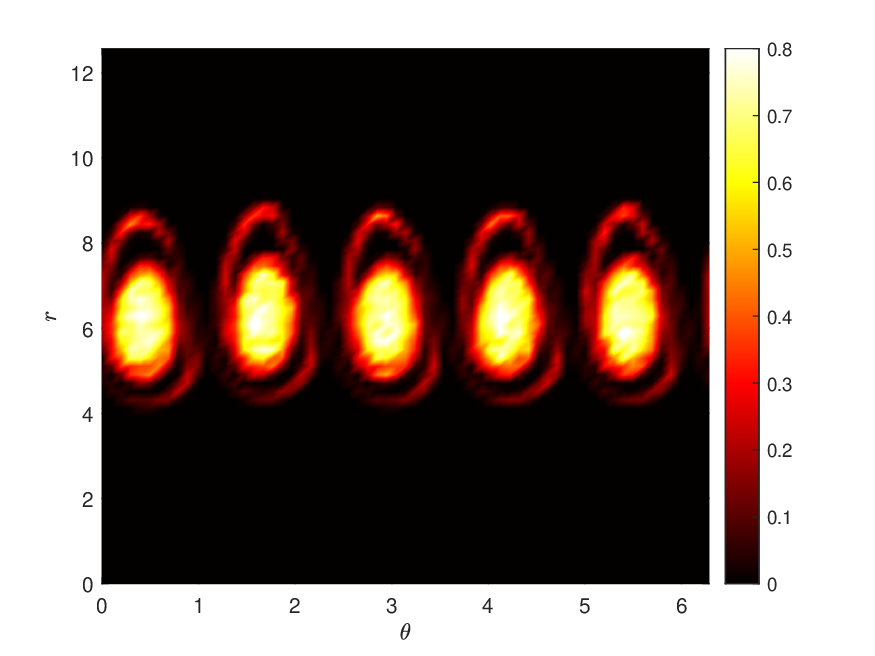}
}
\subfigure{
\includegraphics[scale=.33]{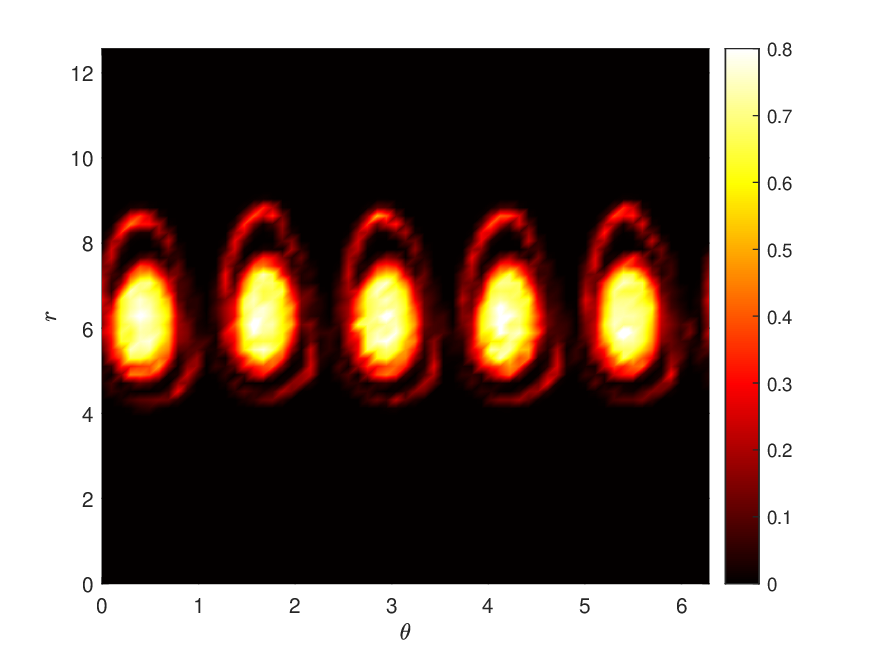}
}
\subfigure{
\includegraphics[scale=.33]{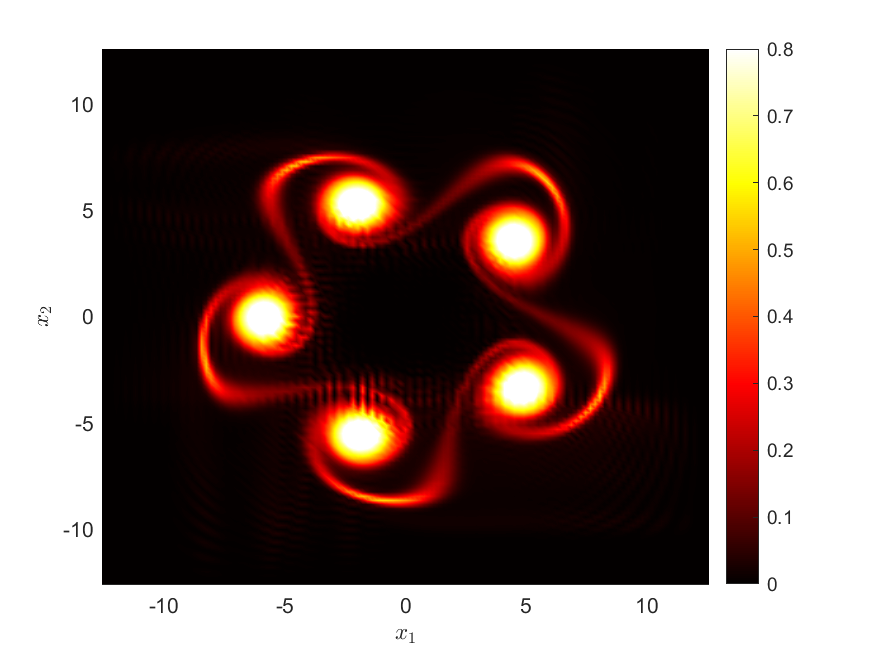}
}
\caption{Time evolution of the density for time $t=10,20,30$ and $l=5$ with parameter $\varepsilon=0.01$. Left column: $r\tilde{\rho}$; Middle column: $r\tilde{\varrho}$; Right column: $\varrho$. Top row: $t=10$; Middle row: $t=20$; Bottom row: $t=30$.}
\label{imgCC}
\end{figure}

For comparison, we choose $(x_1,x_2)\in[-4\pi,4\pi]\times[-4\pi,4\pi]$ in~\cite{GAJ2022Hamiltonian} and $(r,\theta)\in[r_0,4\pi]\times[0,2\pi]$ for now which means we can significantly reduce the computational scale to conduct numerical experiments with the same physical parameters.
Here, $r_0$ is a positive constant introduced to avoid singularities.
Moreover, the number of particles is $1.6\times 10^7$, and for $r$ and $\theta$, we assume Dirichlet boundary conditions and periodic boundary conditions, respectively.
Uniform grid of $r$-$\theta$ plane also equals to divide the computational domain on a ring in $x$-$y$ plane by the approach showed in Figure~\ref{imggrid}.
Meanwhile, the reduction in the computational domain allows us to employ a coarser finite element mesh, or equivalently, a lower resolution. 
The electric field solver constitutes the most computationally intensive and least parallelizable part of the entire program, hence the adoption of the new coordinate system further accelerates the computation. 
Specifically, we employ $N_r=N_{\theta}=64$ here, which would exhibit insufficient resolution under Cartesian coordinates.
For example, in our calculations of the guiding center system in Cartesian coordinates, a spatial resolution of $N_{x_1}=N_{x_2}=128$ was applied. 
As can be seen in Figure~\ref{imgCC}, some minor numerical noise remains due to insufficient resolution.

In Figure~\ref{imgCC0}, we plot the numerical results for $\varepsilon=1$, and the time step $\Delta{t}=0.1$, where $r_0=1$ and a stopping condition is imposed once particles reach the boundary. 
It can be observed that the magnetic field is too weak to confine the particles, leading to their rapid diffusion to the boundary.
Then in Figure~\ref{imgCC}, we show the time evolution of the density when we take $\varepsilon=0.01$, and $\Delta{t}=0.1$.
The $\tilde{\varrho}$ is the solution to the guiding center approximation in orthogonal curvilinear coordinates which reads
$$\partial_t\tilde{\varrho}+\frac{K\tilde{\bm{E}}}{r}\cdot\nabla_y\tilde{\varrho}=0.$$
where $\tilde{\rho}\to\tilde{\varrho}$ when $\varepsilon\to 0$.
On the other hand, $\varrho$ is the one in Cartesian coordinate.
Despite the low resolution and large time step, the instability phenomenon is still captured well which can be verified through comparison with the results of guiding center model.
They are also consistent with the results in Cartesian coordinates, see the right column. 

The initial distribution is inherently defined in the $r$-$\theta$ coordinate system, hence generating initial particles in this coordinate framework aligns more naturally with mathematical intuition.
At the same time, it can also be seen that the use of the new coordinate makes it easier to observe the confining effect of the magnetic field on the plasma, and we can try to further reduce the computational area by using a larger $r_0$.

\section{Conclusion}
In this work, we have developed the Particle-in-Cell methods for solving the magnetized Vlasov--Poisson system in orthogonal curvilinear coordinates.
In numerical computation, we have used the finite element method in space and the semi-implicit method in time.
Also, asymptotic preservation property of the full discretization method is verified.
This guarantees the numerical simulation over long-time.
We present a $2+2$-dimensional example for application, in this example the external magnetic field is strong and measured by $1/\varepsilon$.
It have been studied that our numerical methods can accurately portray the physical phenomena, and demonstrate the confinement affect of the magnetic field on the plasma.
By choosing the appropriate coordinates for the given numerical experiments, we can reveal the conservative property of physical quantity and also can reduce the computation cost in numerical simulation.

After spatial discretization in curvilinear coordinates, we can get the semi-discrete Hamiltonian systems.
However, the semi-implicit algorithms, although asymptotically stable, generally do not preserve the discrete Poisson structure.
Thus the better temporal methods that can preserve the geometric structure and also easy to be  implemented, need to be established.
Also the physical problem in more complex electro-magnetic field is also needed to be investigated.
This will be done in our future work.

\section{Acknowledgments}
The authors thank the reviewers for their helpful advice.
This research is supported by the National Natural Science Foundation of China with Grant No. 12271513.

\appendix
\section{Vector fields in orthogonal curvilinear coordinates}\label{app:0}

We assume that $x=(x_1,x_2,x_3)$ is Cartesian coordinate and $y=(y_1,y_2,y_3)$ is an orthogonal coordinate.
Then we have
$$dx_i=\sum\limits_{j} \frac{\partial x_i}{\partial y_j}dy_j.$$
Further, the quadratic form
$$ds^2=\sum\limits_{i}dx_i^2$$ has the following form in curvilinear coordinate
$$ds^2=\sum\limits_{ij} g_{ij}dy_idy_j,$$
where $$g_{ij}=t_i\cdot t_j$$ and $$t_i=\frac{\partial x}{\partial y_i}$$ is called covariant basis.
At the same time, there exist dual contravariant basis $n_1,n_2,n_3$ satisfying $t_i\cdot n_j=\delta_{ij}$.
And since the coordinate system is orthogonal we have
\begin{equation*}
g_{ij}=
\left\{
\begin{aligned}
&g_{ii},\ i=j,\\
&0,\ i\neq j.
\end{aligned}
\right.
\end{equation*}
Denote the positive numbers $H_i=\sqrt{g_{ii}}$ and they are often called Lame coefficients.
If the corresponding map on the domain is a diffeomorphism, its tangent map is an isomorphism.
That is, the standard orthogonal basis of Cartesian coordinate corresponds to a set of covariant basis in curvilinear coordinate, but in general the latter are not unit vectors, and the moduli of these basis vectors are Lame coefficient $H_i$s.
Therefore the unit basis vectors in the curvilinear coordinate will be
$$e_i=\frac{1}{H_i}\frac{\partial x}{\partial y_i}=\frac{1}{H_i}t_i.$$

On the other hand, with these coefficients, we can consider the differential forms of the curvilinear coordinate.
The volume element can be expressed as
$$dV=H_1 H_2 H_3 dy_1 dy_2 dy_3.$$
For a scalar field $f$, the gradient is
\begin{equation*}
\begin{aligned}
\mathrm{gard} f=\nabla f=\sum_{i}\frac{1}{H_i}\frac{\partial f}{\partial y_i}e_i.
\end{aligned}
\end{equation*}
By using bases, the gradient operator can also be expressed as $\nabla=\sum\limits_{i}\frac{\partial}{\partial y_i}n_i$.

For a vector field $B$, the divergence is
\begin{equation*}
\begin{aligned}
\mathrm{div} B=\nabla\cdot B=\frac{1}{H_1 H_2 H_3}\sum_{i}\frac{\partial}{\partial y_i}(H_j H_k B_i).
\end{aligned}
\end{equation*}

For a scalar field $f$, the Laplace operator is
\begin{equation*}
\begin{aligned}
\Delta{f}=\mathrm{div}(\mathrm{gard} f)=\frac{1}{H_1 H_2 H_3}\sum_{i}\frac{\partial}{\partial y_i}\left(\frac{H_j H_k}{H_i} \frac{\partial f}{\partial y_i}\right),
\end{aligned}
\end{equation*}
where $i,j,k\in\{1,2,3\}$ and are different from each other.

\begin{example}{Cylindrical coordinate}

We consider the cylindrical coordinate $$(r,\theta,z)$$ with the mapping
$$x=r\cos{\theta},\ y=r\sin{\theta},\ z.$$
By defining covariant basis as
\begin{align*}
&t_1=\frac{\partial (xe_1+ye_2+ze_3)}{\partial r}=\cos{\theta}e_1+\sin{\theta}e_2=e_r,\\
&t_2=\frac{\partial (xe_1+ye_2+ze_3)}{\partial \theta}=-r\sin{\theta}e_1+r\cos{\theta}e_2=re_{\theta},\\
&t_3=\frac{\partial (xe_1+ye_2+ze_3)}{\partial z}=e_3=e_z,
\end{align*}
it is straightforward to derive the contravariant basis as
\begin{equation*}
n_1=e_r,\ n_2=\frac{1}{r}e_{\theta},\ n_3=e_z,
\end{equation*}
The Lame coefficients are obtained from covariant basis as
$$H_1=1,\ H_2=r,\ H_3=1,$$
Then the quadratic form is
$$ds^2=dr^2+r^2d\theta^2+dz^2,$$
the volume element is
$$dV=rdr d\theta dz,$$
the gradient, divergence and Laplace operator are
\begin{align*}
&\mathrm{grad} f=\nabla f=\frac{\partial f}{\partial r}e_r+\frac{1}{r}\frac{\partial f}{\partial \theta}e_{\theta}+\frac{\partial f}{\partial z}e_z,\\
&\mathrm{div} B=\frac{1}{r}\left(\frac{\partial (r B_r)}{\partial r}+\frac{\partial B_{\theta}}{\partial \theta}+\frac{\partial (r B_z)}{\partial z}\right),\\
&\Delta{f}=\frac{1}{r}\left(\frac{\partial}{\partial r}\left(r \frac{\partial f}{\partial r}\right)+\frac{\partial}{\partial \theta}\left(\frac{1}{r} \frac{\partial f}{\partial \theta}\right)+\frac{\partial}{\partial z}\left(r \frac{\partial f}{\partial z}\right)\right).
\end{align*}
\end{example}

\section{Propositions}\label{app:1}

\begin{proposition}\label{brapro}\cite{Olver1986}
$K(x)$ is the structure matrix for a Poisson bracket if and only if it has the following properties:\\
(1) Skew-symmetry: $$K_{ij}(x)=-K_{ji}(x),\ i,j=1,\cdots,m,$$
(2) Jacobi identity:
\begin{equation*}
\sum_{l=1}^{m}\left(\frac{\partial K_{ij}}{\partial x_l} K_{lk}+\frac{\partial K_{jk}}{\partial x_l} K_{li}+\frac{\partial K_{ki}}{\partial x_l} K_{lj}\right)=0.
\end{equation*}
\end{proposition}

\begin{proposition}\label{mang}
The matrix
\begin{equation*}
\hat{\bm{B}}=\left(
            \begin{array}{ccc}
              0 & b_3 & -b_2 \\
              -b_3 & 0 & b_1 \\
              b_2 & -b_1 & 0 \\
            \end{array}
          \right)
\end{equation*}
with respect to the magnetic field $\bm{B}:=[b_1,b_2,b_3]^\top$ satisfies
\begin{equation*}
\hat{\bm{B}}^3=-b^2\hat{\bm{B}},\quad (\bm{B}\cdot y)\bm{B}=\hat{\bm{B}}^2 y+b^2 y,
\end{equation*}
$\vert \bm{B}\vert=b$ and it holds for any vector $y\in\mathbb{R}^{3}$.
\end{proposition}

\bibliographystyle{plain}
\bibliography{refs}

\end{document}